\documentclass[12pt,reqno]{amsart}
\usepackage{amsmath}
\usepackage{amssymb}
\usepackage{amstext}
\usepackage{a4wide}
\usepackage{graphicx}
\usepackage[numbers,sort&compress]{natbib}
\allowdisplaybreaks \numberwithin{equation}{section}
\usepackage{color}
\usepackage{cases}

\numberwithin{equation}{section}

\newtheorem{theorem}{Theorem}[section]
\newtheorem{proposition}[theorem]{Proposition}
\newtheorem{corollary}[theorem]{Corollary}
\newtheorem{lemma}[theorem]{Lemma}
\newtheorem*{Yudovich's Theorem}{Yudovich's Theorem}

\theoremstyle{definition}

\theoremstyle{remark}
\newtheorem{remark}[theorem]{Remark}

\begin{document}

\title
[An extension of Arnold's second stability theorem]{An extension of Arnold's second stability theorem in a multiply-connected domain}

 \author{Guodong Wang,  Bijun Zuo}
\address{Institute for Advanced Study in Mathematics, Harbin Institute of Technology, Harbin 150001, P.R. China}
\email{wangguodong@hit.edu.cn}
\address{College of Mathematical Sciences, Harbin Engineering University, Harbin {\rm150001}, PR China}
\email{bjzuo@amss.ac.cn}


\begin{abstract}
We give a sufficient condition for the  nonlinear stability of  steady flows of a two-dimensional ideal  fluid in  a  bounded multiply-connected domain, which generalizes a stability criterion  proved by Arnold in the 1960s. The most important ingredient of the proof is to establish a variational characterization for the steady flow under consideration, which is achieved based on the energy-Casimir method proposed by Arnold, and the supporting functional method introduced by Wolansky and Ghil.
Nonlinear stability then follows from a compactness argument related to the variational characterization and proper use of conserved quantities of the two-dimensional Euler equations.

\end{abstract}

\maketitle

\tableofcontents

\section{Introduction and main result}

In the 1960s, Arnold  \cite{A1,A2} proved two general theorems on the nonlinear stability of steady flows of a two-dimensional ideal fluid, now usually referred to as Arnold's first and second stability theorems in the literature.
Roughly speaking, Arnold's stability theorems  asserted that  a steady Euler flow is nonlinearly stable in a certain sense if its stream function $\bar\psi$ and its vorticity $\bar\omega$ satisfies
\begin{equation}\label{arnc1}
-c_{1}\leq \frac{\nabla \bar\omega}{\nabla\bar\psi}\leq -c_{2}
\end{equation}
for some positive constants $c_{1}$ and $c_{2}$ (the first stability theorem), or
\begin{equation}\label{arnc2}
c_{3}\leq \frac{\nabla \bar\omega}{\nabla\bar\psi}\leq c_{4}
\end{equation}
for some positive constants $c_{3}$ and $c_4$ with $c_4$ being sufficiently small (the second stability theorem). See also \cite{AK}, Chapter 2.  Our main purpose in this paper is to give an extension of the second stability theorem in the case of a multiply-connected bounded domain, showing that  the condition \eqref{arnc2} can be relaxed and the stability conclusion can be strengthened.

  \subsection{2D Euler equations in a multiply-connected domain}

We start with the following two-dimensional (2D) incompressible Euler equations
\begin{equation}\label{eul}
\begin{cases}
\partial_{t}\mathbf v+(\mathbf v\cdot\nabla)\mathbf v=-\nabla P,\\
\nabla\cdot\mathbf v=0,
\end{cases}
\end{equation}
 describing  the motion of an inviscid planar fluid with unit density.
Here $\mathbf v=(v_{1},v_{2})$ is the velocity field, and $P$ is the scalar pressure.
We assume that the fluid occupies a multiply-connected bounded  domain $D$, with the form
\begin{equation}\label{dform}
D=D_0\setminus \cup_{i=1}^{N}D_{i},\quad \bar D_{i}\subset D_{0},\,\,i=1,\cdot\cdot\cdot,N,
\end{equation}
where each $D_{i}$ is a simply-connected bounded domain of $\mathbb R^{2}$ with a smooth boundary $\Gamma_{i}$, and $\bar D_{i}\cap \bar D_{j}=\varnothing$ whenever $i\neq j$.
 Denote by $\mathbf n$ the outward (with respect to $D$) unit normal vector on $\partial D$. We impose the following boundary condition for \eqref{eul}:
\begin{equation}\label{bdry}
\mathbf v\cdot\mathbf n|_{\partial D}=0,\quad t\geq0.
\end{equation}

If we identify  $\mathbf v$ with the three-dimensional vector $(v_{1},v_{2},0)$, then   the corresponding vorticity vector curl$\mathbf v$ is $(0, 0, \partial_1v_2-\partial_2v_1)$. Define
 \begin{equation}\label{svord}
\omega=\partial_{1}v_{2}-\partial_{2}v_{1},
\end{equation}
called the \emph{scalar vorticity}.

For simplicity, we mainly consider \emph{classical solutions} of \eqref{eul} in this paper. By a classical solution, we mean a pair of functions  $(\mathbf v, P)$ such that  $\mathbf v\in C^{1}([0,+\infty)\times\bar D), $  $P(t,\cdot)\in C^{1}(\bar D) $ for any $t\geq 0,$
and $(\mathbf v,P)$ satisfies  the Euler system \eqref{eul} and the boundary condition \eqref{bdry}  pointwise.  Given a sufficiently regular initial velocity field $\mathbf v(0,\cdot)=\mathbf v_{0}$ with $\nabla\cdot \mathbf v_{0}=0,$  a classical solution always exists and is unique. See \cite{MB}, Chapter 3, or \cite{MPu}, Chapter 2.

For classical solutions, the following conservation properties hold, which play an important role in this paper.
\begin{itemize}
\item [(C1)] Conservation of kinetic energy:
\[\frac{1}{2}\int_{D}|\mathbf v(t,x)|^{2}dx=\frac{1}{2}\int_{D}|\mathbf v(0,x)|^{2}dx,\quad \forall\,t>0.\]
\item [(C2)] Conservation of circulation (also called Kelvin's circulation theorem):
 \[ \int_{\Gamma_{i}}\mathbf v(t,x)\cdot d\mathbf r=\int_{\Gamma_{i}}\mathbf v(0,x)\cdot d\mathbf r,\quad \forall\,t>0,\,\,i=0,1,\cdot\cdot\cdot,N.\]
Here and henceforth, $\mathbf r$ denotes the anticlockwise rotation of the outward  unit normal vector  $\mathbf n$  through an angle of $\pi/2$.
 \item [(C3)] Conservation of vorticity:
  \[ \omega(t,\cdot)\in\mathcal R_{\omega(0,\cdot)},\quad \forall\,t>0.\]
  Here, for some $w\in L^1_{\rm loc}(D)$,  $\mathcal R_{w}$ denotes the set of  rearrangements  of  $w$ on $D$, i.e.,
\begin{equation}\label{rwdef}
\mathcal R_{w}=\left\{v\in L^{1}_{\rm loc}(D)\mid |\{x\in D\mid v(x)>s\}|=|\{x\in D\mid w(x)>s\}|,\,\forall\,s\in\mathbb R\right\},
\end{equation}
where $|\cdot|$ is the two-dimensional Lebesgue measure.
  \end{itemize}
   The proofs of (C1)-(C3) can be found in \cite{MPu}, Chapter 1.
Note that the conservation of vorticity implies that for any $f\in C(\mathbb R)$,
  \begin{equation}\label{brol}
  \int_{D}f(\omega(t,x))dx=\int_{D}f(\omega(0,x))dx,\quad \forall\,t>0.
  \end{equation}

In view of  the divergence-free condition $\nabla\cdot\mathbf v=0$ and the impermeability boundary condition \eqref{bdry}, there is a scalar function $\psi: [0,+\infty)\times \bar D\to\mathbb R$, called the \emph{stream function}, such that
\begin{itemize}
\item[(i)] $\psi(t,\cdot)\in C^{2}(\bar D)$ for any $t\geq 0;$
\item[(ii)] $\psi$ is constant on $\Gamma_{i},$ $i=0,1,\cdot\cdot\cdot,N;$
\item[(iii)] $\mathbf v$ can be expressed as
\begin{equation}\label{vexp1}
\mathbf v=\nabla^{\perp}\psi,\quad \nabla^{\perp}:=(\partial_{2},-\partial_{1}).
\end{equation}
\end{itemize}
See for example \cite{MPu}, p. 16. Throughout this paper,  we always assume that
\begin{equation}\label{czy}
\mbox{\emph{the stream function vanishes on  $\Gamma_{0}$}}
\end{equation}
  by adding a suitable constant to it.

Below we show that the stream function $\mathbf \psi$  can be uniquely determined by the vorticity $\omega$ and the \emph{circulation vector} $\mathbf a=(a_{1},\cdot\cdot\cdot,a_{N})\in\mathbb R^{N},$ where $a_{i}$ is the circulation of the velocity field around the $i$-th inner boundary component $\Gamma_{i},$ i.e.,
 \begin{equation}\label{cvec1}
 a_{i}=\int_{\Gamma_{i}}\mathbf v\cdot d\mathbf r,
 \end{equation}
 Note that the integral in \eqref{cvec1} is constant in time by Kelvin's circulation theorem.
In view of  \eqref{svord}, \eqref{vexp1}, \eqref{czy} and \eqref{cvec1},
$\psi$  satisfies the following elliptic problem
\begin{equation}\label{stb02}
\begin{cases}
-\Delta \psi=\omega&\mbox{in }D,\\
\psi=0&\mbox{on }\Gamma_{0},\\
\psi\mbox{ is contant on }\Gamma_{i},&i= 1,\cdot\cdot\cdot,N,\\
\int_{\Gamma_{i}}\nabla\psi\cdot d\mathbf n=-a_{i},&i= 1,\cdot\cdot\cdot,N.
\end{cases}
 \end{equation}
 To express $\psi$ in terms of $\omega$ and $\mathbf a$, we introduce the following notation.
 \begin{itemize}
 \item [(N1)]Denote by $\mathcal G$  the Green operator related to $-\Delta$ in $D$ with zero Dirichlet boundary condition; or equivalently,  $\mathcal G\phi$ is the unique solution of the following elliptic problem
\begin{equation}\label{biosa10}
\begin{cases}
-\Delta \mathcal G\phi =\phi&\mbox{in } D,\\
\mathcal G\phi=0 &\mbox{on }D.
 \end{cases}
\end{equation}
 \item  [(N2)] For $1\leq i\leq N,$ let $\zeta_{i}\in C^{\infty}(\bar D)$ be the solution of the following Laplace equation
\begin{equation}\label{zet1n}
 \Delta \zeta_{i}=0\,\,\,\,\mbox{in } D,\quad
\zeta_{i}|_{\Gamma_{j}}=
\begin{cases}
1&\mbox{if }j=i,\\
0&\mbox{if }j\neq i,\, j=0,1,\cdot\cdot\cdot,N.
\end{cases}
\end{equation}
 \item  [(N3)]Define
\begin{equation}\label{zet2n}
p_{ij}=\int_{D}\nabla\zeta_{i}\cdot\nabla\zeta_{j}dx,\quad 1\leq i,j\leq N.
\end{equation}
It is easy to check that $(p_{ij})$ is a symmetric and positive definite matrix.
 \item [(N4)] The inverse  of $(p_{ij})$ is denoted by $ (q_{ij})$, which is also symmetric and positive definite.
 \item[(N5)] Define
 \begin{equation}\label{biosaa0}
\mathcal P\omega=\mathcal G\omega+\sum_{i,j=1}^{N}q_{ij}\int_{D}\zeta_{i}\omega dx\zeta_{j},\quad h_{\mathbf a}=-\sum_{i,j=1}^{N}q_{ij}a_{i}\zeta_{j}.
\end{equation}
 \end{itemize}
By Proposition \ref{propc2}, the elliptic problem \eqref{stb02} has a unique solution
  \begin{equation}\label{biosaa}
  \psi=\mathcal P\omega+h_{\mathbf a}.
 \end{equation}
As a consequence, the velocity field $\mathbf v$ can be expressed in terms of $\omega$ and $\mathbf a$ as
   \begin{equation}\label{biosab}
  \mathbf v=\nabla^{\perp}(\mathcal P\omega+h_{\mathbf a}).
  \end{equation}
The relation \eqref{biosab} is usually called the \emph{Biot-Savart law}, since it is very similar to the way that the magnetic field is generated by the electric current.

 According to the above discussion, at any fixed time, the state of an ideal fluid in $D$ can be described in three ways: the velocity field $\mathbf v$, the stream function $\psi$, or the pair $(\omega,\mathbf a),$ where $\omega$ is the vorticity and $\mathbf a$ is the circulation vector given by \eqref{cvec1}.

 \subsection{Arnold's second stability theorem}

For an Euler flow, if the velocity $\mathbf v$ and the scalar pressure $P$ do not depend on the  time variable, we call it a \emph{steady Euler flow}.

A large class of  steady Euler flows are those related to the following elliptic problem:
\begin{equation}\label{semim}
\begin{cases}
-\Delta  u=g(u),&\mbox{in } D,\\
u=0,&\mbox{on }\Gamma_{0},\\
u \mbox{ is constant on }\Gamma_{i},\,\, i=1,\cdot\cdot\cdot,N.
\end{cases}
\end{equation}
In fact, if $\bar\psi\in C^{2}(\bar D)$ solves \eqref{semim} with $g\in C(\mathbb R)$, then the reader can check that
 \[ {\mathbf v}=\nabla^{\perp}\bar\psi,\quad  P=-G(\bar\psi)-\frac{1}{2}|\nabla\bar\psi|^{2} \]
satisfy \eqref{eul} and \eqref{bdry},
where $G$ is an antiderivative of $g$.

Given a steady Euler flow, an important problem is to study its \emph{nonlinear stability}. More specifically, if another Euler flow (not necessarily steady) is ``close'' to the steady flow at $t=0$, we want to know  whether
it remains ``close'' to the steady flow for all $t>0$. Of course, the norm used to measure the ``closeness" of two flows and  the class of perturbed flows should be specified in rigorous statements.

Our purpose in this paper is to give some general and natural conditions to ensure the nonlinear stability of  steady Euler flows related to \eqref{semim}.
Our main result is closely related to Arnold's second stability theorem, the precise statement  of which in the setting of a multiply-connected domain is given below.

To make the statement concise, define
\begin{equation}\label{defx}
\mathcal X=\left\{u\in H^{1}(D)\mid u=v+\sum_{i=1}^{N}\theta_{i}\zeta_{i},\,\, v\in H^{1}_0(D), \,\,\theta_{i}\in\mathbb R, \,\,i=1\cdot\cdot\cdot,N\right\},
\end{equation}
\begin{equation}\label{defy}
\mathcal Y=\left\{u\in H^{2}(D)\mid u\in\mathcal X,\,\, \int_{\Gamma_{i}}\nabla u\cdot d\mathbf n=0, \,\,i=1\cdot\cdot\cdot,N\right\}.
\end{equation}
Note that these two function spaces will be frequently used throughout this paper. Some of their properties are given in Appendix \ref{apdx2}.

\begin{theorem}[Arnold's second stability theorem, \cite{A1,A2,AK}]\label{thm0}
Assume that $D\subset\mathbb R^{2}$ is a multiply-connected bounded  domain of the form \eqref{dform}, and  $\bar\psi\in C^{2}(\bar D)$ solves \eqref{semim} with $g\in C^{1}(\mathbb R)$. Suppose
\begin{equation}\label{min0}
 \min_{\bar D}g'(\bar\psi)>0,
 \end{equation}
 \begin{equation}\label{max0}
 \max_{\bar D}g'(\bar\psi)<\mathsf c,
   \end{equation}
   where $\mathsf c$ is a  positive number such that
   \begin{equation}\label{mathsfc}
   \mathsf c\int_D|\nabla u|^2dx\leq \int_D|\Delta u|^2dx,\quad \forall\,u\in\mathcal Y.
   \end{equation}
Let $\psi(t,x)=\bar\psi(x)+\varphi(t,x)$ be  the stream function of another  Euler flow such that 
\begin{equation}\label{jjzuo150}
\varphi(t,\cdot)\in\mathcal Y,\quad \forall\,t\geq 0,
\end{equation} 
Then
 \begin{equation}\label{conc3}
 \left(\frac{1}{ \max_{\bar D}g'(\bar\psi)}-\frac{1}{\mathsf c}\right)\|\Delta\varphi(t,\cdot)\|_{L^{2}(D)}^{2} \leq \frac{1}{ \min_{\bar D}g'(\bar\psi)}\|\Delta\varphi(0,\cdot)\|_{L^{2}(D)}^{2},\,\,\forall\, t>0.
\end{equation}
As a consequence,  the steady flow with  stream function $\bar\psi$ is nonlinearly stable in the following sense:

For any $\varepsilon>0,$ there exists some $\delta>0$, such that for any   Euler flow with  stream function $\bar\psi+\varphi$  satisfying
\begin{equation}\label{jjzuo15}
\varphi(t,\cdot)\in\mathcal Y,\quad \forall\,t\geq 0,
\end{equation}
it holds that
\[\|\Delta\varphi(0,\cdot)\|_{L^{2}(D)}<\delta\quad\Longrightarrow\quad \|\Delta\varphi(t,\cdot)\|_{L^{2}(D)}<\varepsilon \quad\forall\,t>0.\]

\end{theorem}

\begin{remark}
The assumptions \eqref{min0} and \eqref{max0}  correspond to the condition \eqref{arnc2}.
\end{remark}

\begin{remark}
   Note that there exists a positive number $\mathsf c$ such that \eqref{mathsfc} holds by Proposition \ref{jjzuo} in Section 3.
\end{remark}

\begin{remark}
The assumptions \eqref{jjzuo150} and \eqref{jjzuo15} are  reasonable by Kelvin's circulation theorem.
\end{remark}

Arnold first stated  Theorem \ref{thm0}  in \cite{A1}, and then gave a rigorous proof in \cite{A2}.
 The above (more detailed) version is taken from Arnold and Khesin's book \cite{AK}, p. 97.  Since the proof of Theorem \ref{thm0} was omitted in \cite{AK},   we provide a detailed one  in Appendix \ref{apdx1}  for the reader's convenience.

Theorem \ref{thm0} gives some succinct conditions that are sufficient for the nonlinear stability of steady Euler flows related to \eqref{semim}. There are, however, some limitations in specific applications:
\begin{itemize}
\item[(i)] The conditions \eqref{min0} and \eqref{max0} are very strong requirements on the function $g$, which exclude many interesting cases.
A typical example is as follows. Let $\tilde u$ be a minimizer for the variational problem \eqref{var3} in Section 3, then $\tilde u$ corresponds to a steady Euler flow by Proposition \ref{prop1}(ii). Such a steady flow is nonlinearly stable (which can be proved directly by using Proposition \ref{prop37}), but the condition \eqref{max0} is not satisfied, i.e., a positive constant $\mathsf c$ satisfying \eqref{max0} and \eqref{mathsfc} does not exist (see  Proposition \ref{prop37}).

\item[(ii)] The condition \eqref{jjzuo15} in Theorem \ref{thm0} requires
 that the perturbed flow (with  stream function $\psi$) has the same circulation around every inner boundary component of the domain $D$ as the original flow (with stream function $\bar\psi$), which is not very natural.
  \item[(iii)] The norm to measure nonlinear stability is the $L^2$ norm of the vorticity. For other norms, say $L^p$ norms of the vorticity, whether nonlinear stability holds is not clear.
  \end{itemize}

\subsection{Main result}

Our main result in this paper is the following extension of Theorem \ref{thm0}.

\begin{theorem}\label{thm1}
Assume that $D\subset\mathbb R^{2}$ is a multiply-connected bounded  domain of the form \eqref{dform}, and  $\bar\psi\in C^{2}(\bar D)$ solves \eqref{semim} with $g\in C^{1}(\mathbb R)$. Suppose
\begin{equation}\label{min0w}
 \min_{\bar D}g'(\bar\psi)\geq 0,
 \end{equation}
 \begin{equation}\label{ssmmy}
 \int_{D}|\nabla u|^{2}-g'(\bar\psi)u^{2}dx\geq 0,\quad\forall\,u\in \mathcal Y.
 \end{equation}
Then for any fixed $1<p<+\infty$, the steady Euler flow related to $\bar\psi$ is nonlinearly stable in the following sense:

For any $\varepsilon>0$, there exists some $\delta>0,$ such that for any  Euler flow with stream function $\bar\psi+\varphi$, it holds that
\[ \|\varphi(0,\cdot)\|_{W^{2,p}(D)}<\delta\quad\Longrightarrow\quad\|\varphi(t,\cdot)\|_{W^{2,p}(D)}<\varepsilon \quad \forall\,t>0.\]
\end{theorem}

 \begin{remark}\label{rafthm1}
By   Corollary \ref{glav2022} (taking $c=-g'(\bar\psi)$ therein),  the condition \eqref{ssmmy} is in fact equivalent to
 \[\int_{D}|\nabla u|^{2}-g'(\bar\psi)u^{2}dx\geq 0,\quad\forall\,u\in \mathcal X.\]
\end{remark}

\begin{remark}\label{vico1}
By  Lemma \ref{propd1}, the nonlinear stability of the steady Euler flow in Theorem \ref{thm1} can also be described as follows.
\begin{itemize}
\item[(i)] Stability in terms of the velocity: Let $\bar{\mathbf v}$ be the corresponding velocity of the steady flow, then for any $\varepsilon>0$, there exists some $\delta>0$, such that for any Euler flow with  velocity $\mathbf v$, it holds that
\[ \|\mathbf v(0,\cdot)-\bar{\mathbf v}\|_{W^{1,p}(D)}<\delta\quad\Longrightarrow\quad\|\mathbf v(t,\cdot)-\bar{\mathbf v}\|_{W^{1,p}(D)}<\varepsilon \quad \forall\,t>0.\]
\item[(ii)]Stability in terms of the  vorticity and the circulation vector: Let $\bar{\omega}$ be the corresponding vorticity and $\mathbf a$ be the corresponding circulation vector  of  the steady flow, then for any $\varepsilon>0$, there exists some $\delta>0$, such that for any Euler flow with  vorticity $\omega$ and circulation vector $\mathbf b$, it holds that
\[ \|\omega(0,\cdot)-\bar{\omega}\|_{L^{p}(D)}+|\mathbf b-\mathbf a|<\delta\quad\Longrightarrow\quad \|\omega(t,\cdot)-\bar{\omega}\|_{L^{p}(D)}<\varepsilon \quad \forall\,t>0.\]
\end{itemize}
 \end{remark}

\begin{remark}
 Theorem \ref{thm1} may hold for less regular perturbations as well. See Remark \ref{lessregular} for a brief discussion.

\end{remark}

 Our Theorem \ref{thm1} extends Arnold's second stability theorem in the following three aspects.
First, the condition \eqref{min0w} is weaker than \eqref{min0}.
Second, the conclusion of Theorem \ref{thm1} is stronger than that of Theorem \ref{thm0}: the perturbed flows in Theorem \ref{thm1} are not required to have the same circulations around the $N$ inner boundaries as the original steady flow; the stability is measured in more general norms (by  Lemma \ref{propd1}, $\|\Delta\varphi\|_{L^2(D)}$ is an equivalent norm to $\|\varphi\|_{H^2(D)}$ for  $\varphi\in\mathcal Y$).
Third,  the condition \eqref{ssmmy} is weaker  than \eqref{max0}. To see this, we show that   \eqref{max0} in fact implies \eqref{ssmmy}. Assume that \eqref{max0} holds.  By \eqref{mathsfc},   the positive constant $\mathsf c$ satisfies
\[\frac{1}{\mathsf c}\geq \sup\left\{\int_D|\nabla u|^2dx\mid  u\in\mathcal Y, \,\,\|\Delta u\|_{L^2(D)}=1 \right\}.\]
Taking into account \eqref{kzj8}, \eqref{var10}  and   Proposition \ref{prop37}(i) in Section 3, we have
\[\mathsf c\leq  \inf\left\{\int_{D}|\nabla u|^{2}dx\mid u\in\mathcal Y, \,\,\|u\|_{L^{2}(D)}=1\right\},\]
which implies
\[\int_D|\nabla u|^2-\mathsf c\int_Du^2 dx\geq 0,\quad\forall\,u\in\mathcal Y.\]
Taking into account \eqref{max0}, we get  \eqref{ssmmy}. Roughly speaking, \eqref{ssmmy} only requires $g'(\bar\psi)$ to be small in an ``average'' sense, rather than in the  $L^{\infty}$ norm as in \eqref{max0}.

By Theorem \ref{thm1}, the steady flow with stream function $\tilde u,$ a minimizer for the variational problem \eqref{var3} in Section 3, is nonlinearly stable. In addition,  Theorem \ref{thm1} also applies to a large class of steady Euler flows obtained by the variational approach. Consider the following functional
\[\mathcal I(u)=\frac{1}{2}\int_D|\nabla u|^2dx-\int_DG(u)du,\quad G(s)=\int_0^sg(\tau)d\tau.\]
If $g$ satisfies some suitable growth conditions, then $I$ is well-defined in $\mathcal X$, and has a global minimum, say $\bar\psi$, in $\mathcal X$. See for example \cite{Bad}, Theorem 2.1.5. The fact that the first variation of $\mathcal I$ at $\bar\psi$ vanishes gives
\[-\Delta\bar\psi=g(\bar\psi),\]
and the fact that the second variation of $\mathcal I$ at $\bar\psi$ is nonnegative gives
 \begin{equation}\label{my990}
 \int_{D}|\nabla u|^{2}-g'(\bar\psi)u^{2}dx\geq 0,\quad\forall\,u\in \mathcal X.
 \end{equation}
 Then by Theorem \ref{thm1}, $\bar\psi$ corresponds to a stable steady Euler flow.

The proof of Theorem \ref{thm1} is based on the Lyapunov direct method. It consists of three steps:
\begin{itemize}
\item [(i)] a suitable variational characterization,
\item [(ii)] a compactness result related to the variational characterization,
\item [(iii)] proof of stability based on conservation properties of the Euler equations.
\end{itemize}
  These three ingredients are also essential in studying the nonlinear stability of other steady Euler flows. See for example \cite{Abe,BG,BD,Bjde,Bcmp,CWW,CWCV,CWN,CD,Ta,WGu2,WP}.

The variational characterization, formulated as Theorem \ref{thm2} in  Section \ref{varcha}, is the most important ingredient of the proof. Therein, we show that the steady flow in Theorem \ref{thm1} has \emph{strict} local maximum kinetic energy among all \emph{isovortical} flows that have the same circulations around every  boundary component.  Here a family of instantaneous flows are said to be isovortical if their vorticities are rearrangements of some fix function. Such a  characterization is very much in the spirit of the method of Arnold \cite{A1,A11,A2}. See also \cite{B5,Bjde,Bcmp}.

To prove the variational characterization, we first show that the nonnegative condition \eqref{ssmmy} actually implies certain weak positive definiteness (Proposition \ref{keyprop}). This is achieved by contradiction in combination with a weak convergence argument.
With weak positive definiteness, we can use the energy-Casimir method proposed by Arnold \cite{A1,A2}, and the supporting functional method introduced by Wolansky and Ghil \cite{WG0,WG}, to prove the desired  variational characterization.

Note that in the case of a simply-connected domain, a similar result to Theorem \ref{thm2} has been obtained by the first author in \cite{WGu1}. Compared with \cite{WGu1}, the presence of  inner boundaries in this paper causes some essential difficulties, and new groundwork needs to be laid. A noteworthy difference appears in the proof of weak positive definiteness. Therein, a necessary step is to show that the condition \eqref{ssmmy} is ``closed" under weak convergence in $H^1(D)$. This can not be proved by standard density argument since $\mathcal Y$ is not closed in $H^1(D)$. Note that this problem does not appear in \cite{WGu1}, since  $\mathcal Y$ becomes $H^1_0(D)$ without inner boundary components. To overcome this difficulty,  we employ an indirect approach; more precisely, we prove by studying a constrained variational problem  in  $\mathcal X$ that
\begin{equation}\label{caxy}
\inf_{u\in\mathcal X, \|u\|_{L^{2}(D)}=1} \int_{D}|\nabla u|^{2}-g'(\bar\psi)u^{2}dx =\inf_{u\in\mathcal Y,  \|u\|_{L^{2}(D)}=1} \int_{D}|\nabla u|^{2}-g'(\bar\psi)u^{2}dx.
\end{equation}
See Proposition \ref{prop1}(iii). Based on this fact,  \eqref{ssmmy} implies that
 \begin{equation}\label{my99m}
 \int_{D}|\nabla u|^{2}-g'(\bar\psi)u^{2}dx\geq 0,\quad\forall\,u\in \mathcal X.
 \end{equation}
Therefore the condition \eqref{ssmmy} is indeed ``closed" under weak convergence in $H^1(D)$ since $\mathcal X$ is closed in $H^1(D)$.

Regarding the variational characterization in Section 4, we prove a compactness result in Section 5 based on the positive definiteness of the operator $\mathcal P$  and  two properties of rearrangements (Lemma \ref{lem201} and Lemma \ref{lem202}).
Then  nonlinear stability is proved in Section 6 by a standard contradiction argument in combination of the energy and vorticity conservation of the 2D Euler equations. Sections 5 and 6 are mostly inspired by Burton \cite{B5}.

Except for the Lyapunov direct method, another very effective approach to study the stability/instability of 2D steady Euler flows is the linearized method. We refer the interested reader to \cite{BG,FH,FSV,G,K,LZ1,LZ,LZ3,VF} and the references therein.

The rest of paper is organized as follows. In Section 2, we give some preliminary materials. In Section 3, we study several constrained variational problems in $\mathcal X$ and $\mathcal Y,$ which are used frequently in this paper.  Sections 4-6 are devoted to the proof of Theorem \ref{thm1}.

\section{Preliminaries}


We first list some properties of increasing functions  defined on the real line, which will be used in  Section \ref{varcha}.

\begin{lemma}[\cite{WGu1}, Lemma 2.4]\label{mM}
Let $g\in C^{1}(\mathbb R)$ be increasing (i.e., $g(s_{1})\leq g(s_{2})$ whenever $s_{1}\leq  s_{2}$).  Then for any  $-\infty<a<b<+\infty$,   there exists   $\tilde g\in C^{1}(\mathbb R)$ such that
\begin{itemize}
\item[(i)]   $\tilde g(s)=g(s)$ for any $s\in [a,b]$;
\item [(ii)] $\tilde g$ is strictly increasing in $(-\infty, a]$ and $[b,+\infty)$;
\item[(iii)] there exist positive constants $c_1,c_2$ such that
\begin{equation}\label{lgro}
\lim_{s\to+\infty}\frac{\tilde g(s)}{s}=c_1,\quad \lim_{s\to-\infty}\frac{\tilde g(s)}{s}=c_2.
\end{equation}
\end{itemize}

\end{lemma}

\begin{lemma}[\cite{WGu1}, Lemma 2.3]\label{lt}
Let  $g\in C^1(\mathbb R)$ be increasing.   Suppose there exist positive constants $c_1,c_2$ such that \[\lim_{s\to+\infty}\frac{g(s)}{s}=c_1,\quad \lim_{s\to-\infty}\frac{g(s)}{s}=c_2.\]
Denote $G(s)=\int_0^sg(\tau)d\tau$. Define the Legendre transform $\hat G:\mathbb R\to(-\infty,+\infty]$ of $G$ as follows
\begin{equation}\label{dolt}
\hat G(s)=\sup_{\tau\in\mathbb R}(\tau s-G(\tau)).
\end{equation}
Then the following two assertions hold:
\begin{itemize}
\item[(i)] For any $s,\tau\in\mathbb R,$ it holds that
\[\hat G(s)+G(\tau)\geq s\tau,\]
 and the equality holds if and only if $g(\tau)=s.$
\item [(ii)]  Define
\[f(s)=\inf\{\tau\mid g(\tau)=s\}.\]
Then $f$ is a strictly increasing real-valued function on $\mathbb R$,
and
\[\hat G(s)=\int_0^sf(\tau)d\tau+\hat G(0),\quad\forall\,s\in\mathbb R.\]
 As a consequence,
$\hat G$ is locally Lipschitz continuous on $\mathbb R$ and $\hat G'(s)=f(s)$ a.e. $s\in\mathbb R.$
\end{itemize}
\end{lemma}


We recall two results from Burton's paper \cite{B1} concerning the set of rearrangements of a fixed function. Recall the definition \eqref{rwdef}.

\begin{lemma}[\cite{B1}, Theorem 6]\label{lem201}
Let $1< p< +\infty$ be fixed, and $\mathcal R_w$ be set of  rearrangements  of some $w\in L^{p}(D)$  on $D$. Let $\bar {\mathcal R}_w $ be the weak closure of $\mathcal R_{w} $ in $L^p(D).$ Then $\bar {\mathcal R}_{w}$ is convex, i.e., $\theta v_1+(1-\theta)v_2\in \bar{\mathcal R}_w$ whenever $v_1, v_2\in \bar{\mathcal R}_w$ and $\theta\in[0,1].$
\end{lemma}

\begin{lemma}[\cite{B1}, Theorem  4]\label{lem202}
Let $1< p<+\infty$ be fixed. Denote by $p^*$ the H\"older conjugate of $p$, i.e., $p^*=p/(p-1)$.  Let $\mathcal R_{v_0},$ $\mathcal R_{w_0}$ be sets of  rearrangements  of some $v_{0}\in L^p(D)$ and some $w_{0}\in L^{q}(D)$ on $D$, respectively. Then for any $\tilde w\in\mathcal R_{w_0},$ there exists $\tilde v\in \mathcal R_{v_0}$, such that
\[\int_D \tilde v \tilde wdx\geq \int_{D}vwdx,\quad\forall\, v\in {\mathcal R}_{v_0},\,\,w\in {\mathcal R}_{w_0}.\]
\end{lemma}

Lemmas \ref{lem201} and \ref{lem202} will be used in Section \ref{compac}.

For an ideal fluid of unit density in $D$, the kinetic energy   in terms of the velocity $\mathbf v$ is
\[\frac{1}{2}\int_{D}|\mathbf v|^{2}dx.\]
In the following proposition, we show that the kinetic energy  can be expressed in terms of the vorticity and the circulation vector.

   \begin{proposition}\label{propc3}
  Let $D\subset\mathbb R^{2}$ be a  multiply-connected  domain of the form \eqref{dform}.   Consider an Euler flow with velocity $\mathbf v$.
 Let $\omega$ be the corresponding vorticity, and $\mathbf a$ be the corresponding circulation vector, i.e., $\mathbf a=(a_1,\cdot\cdot\cdot,a_N)$ with $a_i$ given by
   \[a_i=\int_{\Gamma_i}\mathbf v\cdot d\mathbf r,\quad i=1,\cdot\cdot\cdot,N.\]
Then the kinetic energy of the fluid can be expressed in terms of $\omega$ and $\mathbf a$ as follows:

\begin{equation}\label{kevv0}
E(\omega,\mathbf a)=\frac{1}{2}\int_{D}\omega\mathcal P\omega dx+\int_{D}h_{\mathbf a}\omega dx+\frac{1}{2}\sum_{i,j=1}^{N}q_{ij}a_{i}a_{j}.
\end{equation}
where $\mathcal P$ and $h_{\mathbf a}$ are given in \eqref{biosaa}.
\end{proposition}

\begin{proof}
By  Proposition \ref{propc2}, we have
\[\mathbf v=\nabla^\perp(\mathcal P\omega+h_{\mathbf a}).\]
Therefore the kinetic energy can be written as
\begin{equation}\label{kevvt}
 \frac{1}{2}\int_{D}|\mathbf v|^{2}dx=\frac{1}{2}\int_{D}|\nabla(\mathcal P\omega+h_{\mathbf a})|^{2}dx.
\end{equation}
We perform the following computation
 \begin{align*}
& \frac{1}{2}\int_{D}|\nabla(\mathcal P\omega+h_{\mathbf a})|^{2}dx \\
= &\frac{1}{2}\int_{D}|\nabla \mathcal P\omega|^{2}dx+\int_{D}\nabla \mathcal P\omega\cdot\nabla h_{\mathbf a} dx+\frac{1}{2}\int_{D}|\nabla h_{\mathbf a}|^{2}dx\\
= &\frac{1}{2}\int_{D}\omega\mathcal P\omega dx+\sum_{i,j=1}^{N}\left(q_{ij}a_{i}\int_{D}\nabla \mathcal P\omega\cdot\nabla \zeta_{j} dx\right)+\frac{1}{2}\int_{D}|\nabla h_{\mathbf a}|^{2}dx\\
= &\frac{1}{2}\int_{D}\omega\mathcal P\omega dx+\sum_{i,j=1}^{N}\left(q_{ij}a_{i}\int_{D}\omega  \zeta_{j} dx\right)+\frac{1}{2}\sum_{i,j,k,l=1}^{N}\left(q_{ij}q_{kl}a_{i}a_{k}\int_{D}\nabla\zeta_{j}\cdot\nabla\zeta_{l}dx\right)\\
= &\frac{1}{2}\int_{D}\omega\mathcal P\omega dx+ \int_{D}h_{\mathbf a}\omega   dx +\frac{1}{2}\sum_{i,j,k,l=1}^{N} p_{jl}q_{ij}q_{kl}a_{i}a_{k}\\
= &\frac{1}{2}\int_{D}\omega\mathcal P\omega dx+ \int_{D}h_{\mathbf a}\omega   dx +\frac{1}{2}\sum_{i,j=1}^{N}  q_{ij}a_{i}a_{j}.
 \end{align*}
Hence the proof is finished.
\end{proof}

Note that $E$, given by \eqref{kevv0}, is well-defined in $L^p(D)\times \mathbb R^N$ for any $1<p<+\infty$. The following two lemmas concerning the weak continuity and local Lipschitz continuity of  $E$ will be used in Sections \ref{compac} and \ref{nonlin}.

\begin{lemma}\label{lipcoe}
Let $1<p<+\infty$ be fixed.
If $w_{n}\rightharpoonup w$ in $L^{p}(D)$ and $\mathbf a_{n}\to\mathbf a$ in $\mathbb R^{N},$ then
\[\lim_{n\to+\infty}E(w_{n},\mathbf a_{n})=E(w,\mathbf a).\]
Here and henceforth, $``\rightharpoonup"$ means weak convergence.
\end{lemma}
\begin{proof}
Denote $\mathbf a_n=(a_{n,1},\cdot\cdot\cdot,a_{n,N})$ and $\mathbf a=(a_1,\cdot\cdot\cdot,a_N).$
Recall the definition  \eqref{kevv0}, we have that
\begin{equation}\label{wce1}
E(w_n,\mathbf a_n)=\frac{1}{2}\int_{D}w_n\mathcal Pw_n dx+\int_{D}h_{\mathbf a_n}w_n dx+\frac{1}{2}\sum_{i,j=1}^{N}q_{ij}a_{n,i}a_{n,j}.
\end{equation}
Since $\mathcal P$ is bounded from $L^{p}(D)$ to $W^{2,p}(D)$ (see Lemma \ref{eiyoyo}(i)), we see $\mathcal Pw_n\rightharpoonup \mathcal Pw$ in $W^{2,p}(D)$ as $n\to+\infty.$ By the fact that $W^{2,p}(D)$ is compactly embedded in $C(\bar D),$  we have that $\mathcal Pw_n\to \mathcal Pw$ in $C(\bar D)$  as $n\to+\infty.$ Therefore as $n\to+\infty$
\begin{equation}\label{wce2}\int_{D}w_n\mathcal Pw_n dx\to \int_{D}w\mathcal Pw dx.
\end{equation}
On the other hand, it is clear that $h_{\mathbf a_n}\to h_{\mathbf a}$ in $C(\bar D)$ as $n\to+\infty$, which implies
\begin{equation}\label{wce3}\lim_{n\to+\infty}\int_{D}h_{\mathbf a_n}w_n dx=\int_{D}h_{\mathbf a}w dx.
\end{equation}
For the third term, it is obvious that
\begin{equation}\label{wce4}
\lim_{n\to+\infty}\sum_{i,j=1}^{N}q_{ij}a_{n,i}a_{n,j}=\sum_{i,j=1}^{N}q_{ij}a_{i}a_{j}.
\end{equation}
 The assertion follows from \eqref{wce2}-\eqref{wce4} immediately.
\end{proof}

\begin{lemma}\label{lipcoe}
Let $1<p<+\infty$ be fixed.
Let $K$ be a bounded set of $L^{p}(D)$, and $I $ be a bounded set of $\mathbb R^{N}$. Then there exists some positive number $C$, depending only on $K, I, p$ and $D$, such that
\begin{equation}\label{lipc3}
|E(v,\mathbf a)-E(w,\mathbf b)|\leq C(\|v-w\|_{L^{p}(D)}+|\mathbf a-\mathbf b|),\quad \forall\, v,w\in K,\,\mathbf a, \mathbf b\in I.
\end{equation}
\end{lemma}
\begin{proof}
Still denote by $p^{*}$   the H\"older conjugate of $p$. By the fact that $\mathcal P$ is bounded from $L^{p}(D)$ to $W^{2,p}(D)$, and the fact that $W^{2,p}(D)$ is embedded in $L^{p^*}(D)$, we deduce that
there exists some $C_{0}>0,$ depending only on  $p$ and $D$, such that
\begin{equation}
\|\mathcal Pv\|_{L^{p^{*}}(D)}\leq C_{0}\|v\|_{L^{p}(D)},\,\,\forall\,v\in L^{p}(D).
\end{equation}
Recalling the definition \eqref{kevv0}, for any $v,w\in K$ and $\mathbf a,\mathbf b\in I$,   we estimate $|E(v,\mathbf a)-E(w,\mathbf b)|$ as follows:
\begin{align*}
&|E(v,\mathbf a)-E(w,\mathbf b)|\\
=&\bigg|\frac{1}{2}\int_{D}v\mathcal Pv-w\mathcal Pwdx+\int_{D}h_{\mathbf a}v-h_{\mathbf b}wdx+\frac{1}{2}\sum_{i,j=1}^{N}q_{ij}(a_{i}a_{j}-b_ib_j)\bigg|\\
=&\bigg|\frac{1}{2}\int_{D}(v-w)\mathcal Pv+(\mathcal Pv-\mathcal Pw)wdx+\int_{D}(h_{\mathbf a}-h_{\mathbf b})v+h_{\mathbf b}(v-w)dx\\
&+\sum_{i,j=1}^Nq_{ij}(a_i+b_i)(a_j-b_j)\bigg|\\
\leq& \frac{1}{2}\|v-w\|_{L^{p}(D)}\|\mathcal Pv\|_{L^{p^*}(D)}+ \frac{1}{2}\|\mathcal Pv-\mathcal Pw\|_{L^{p^{*}}(D)}\|w\|_{L^{p}(D)}
+\|h_{\mathbf a}-h_{\mathbf b}\|_{L^{p^{*}}(D)}\|v\|_{L^{p}(D)}\\
&+\|h_{\mathbf b}\|_{L^{p^{*}}(D)}\|v-w\|_{L^{p}(D)}+\sum_{i,j=1}^N|q_{ij}||a_i+b_i||a_j-b_j|\\
\le&\frac{1}{2}C_{0}(\|v\|_{L^{p}(D)}+\|w\|_{L^{p}(D)}) \|v-w\|_{L^{p}(D)}+\sum_{i,j=1}^{N}\|q_{ij}(a_{i}-b_{i})\zeta_{j}\|_{L^{p^{*}}(D)}\|v\|_{L^{p}(D)}\\
&+\sum_{i,j=1}^{N}\|q_{ij}b_{i}\zeta_{j}\|_{L^{p^{*}}(D)}\|v-w\|_{L^{p}(D)}+\left(\sum_{i,j=1}^N|q_{ij}||a_i+b_i|\right)\left(\sum_{j=1}^N|a_j-b_j|\right).
\end{align*}
Then the claimed estimate \eqref{lipc3} follows from the following obvious facts:
\[\|v\|_{L^p(D)}\leq C_1 \quad \forall\, v\in K,\quad \sum_{i,j=1}^{N}\|q_{ij}(a_{i}-b_{i})\zeta_{j}\|_{L^{p^{*}}(D)} \leq C_{2}|\mathbf a-\mathbf b| \quad \forall\, \mathbf a, \mathbf b\in I.\]
\[\sum_{i,j=1}^{N}\|q_{ij}b_{i}\zeta_{j}\|_{L^{p^{*}}(D)}\leq C_3 \quad \forall\, \mathbf b\in I,\quad\sum_{i,j=1}^Nq_{ij}|a_i+b_i|\leq C_4 \quad \forall \,\mathbf a, \mathbf b\in I.\]
where $C_{1}, C_2, C_3, C_4$ are positive constants depending only on $K, I, p$ and $D$.
\end{proof}

We have shown in Section 1  that the motion of a fluid can be fully described by the velocity field, or the stream function, or the vorticity together with the circulation vector. The following lemma shows that these three descriptions correspond to three equivalent norms when characterizing nonlinear stability.

\begin{lemma}\label{propd1}
Let $D\subset\mathbb R^{2}$ be a bounded multiply-connected smooth domain with the form \eqref{dform}.
Let $\psi\in C^{2}(\bar D)$  satisfy
\begin{equation}\label{semim00}
\begin{cases}
-\Delta \psi=\omega,&x\in D,\\
\psi=0,&x\in\Gamma_{0},\\
\psi \mbox{ is constant on }\Gamma_{i},& i=1,\cdot\cdot\cdot,N,\\
\int_{\Gamma_{i}}\nabla \psi\cdot d \mathbf n=-a_{i},& i=1,\cdot\cdot\cdot,N,
\end{cases}
\end{equation}
where $\omega\in C(\bar D),$ $a_{1},\cdot\cdot\cdot,a_{N}\in \mathbb R.$ Let $\mathbf v=\nabla^{\perp}\psi.$ Then for any $1<p<+\infty,$  it holds that
\begin{equation}\label{appd1}
 \|\psi\|_{W^{2,p}(D)}\lesssim \|\omega\|_{L^{p}(D)}+\sum_{i=1}^{N}|a_{i}|\lesssim  \|\psi\|_{W^{2,p}(D)},
\end{equation}
\begin{equation}\label{appd2}
 \|\psi\|_{W^{2,p}(D)}\lesssim \|\mathbf v\|_{W^{1,p}(D)} \lesssim\|\psi\|_{W^{2,p}(D)}.
\end{equation}
Here $A\lesssim B$ means $A\leq CB$ for some positive constant $C$ depending only on $p$ and $D$.
\end{lemma}

\begin{proof}
We first prove \eqref{appd1}. By  Proposition \ref{propc2} in Appendix \ref{apdx3}, $\psi$ can be expressed in terms of $\omega$ and $\mathbf a$ as follows
\[\psi= \mathcal G\omega+\sum_{i,j=1}^{N}q_{ij}\int_{D}\zeta_{i}\omega dx\zeta_{j}-\sum_{i,j=1}^{N}q_{ij}a_{i}\zeta_{j}.\]
Denote by $p^{*}=p/(p-1)$  the H\"older conjugate of $p$. Applying the H\"older's inequality and standard elliptic estimates, we have that
\begin{equation*}
\begin{split}
\|\psi\|_{W^{2,p}(D)}&\leq \|\mathcal G\omega\|_{W^{2,p}(D)}+\sum_{i,j=1}^{N}|q_{ij}|\|\zeta_{i}\|_{L^{p^*}(D)}\|\omega\|_{L^{p}(D)}\|\zeta_{j}\|_{W^{2,p}(D)}+\sum_{i,j=1}^{N} \|q_{ij}a_{i}\zeta_{j}\|_{W^{2,p}(D)}\\
&\lesssim \|\omega\|_{L^p(D)}+\left(\sum_{i,j=1}^{N}|q_{ij}|\|\zeta_{i}\|_{L^{p^*}(D)}\|\zeta_{j}\|_{W^{2,p}(D)}\right)\|\omega\|_{L^{p}(D)}+\sum_{i,j=1}^{N}\|q_{ij}\zeta_{j}\|_{W^{2,p}(D)}|a_{i}|\\
&\lesssim \|\omega\|_{L^{p}(D)}+\sum_{i=1}^{N}|a_{i}|.
\end{split}
\end{equation*}
On the other hand, applying the H\"older's inequality and the trace inequality (see \cite{LCE}, p. 274), we have that
\begin{equation}\label{appd9}
\begin{split}
\|\omega\|_{L^{p}(D)}+\sum_{i=1}^{N}|a_{i}|&=\|-\Delta\psi\|_{L^{p}(D)}+\sum_{i=1}^{N}\left|\int_{\Gamma_{i}}\nabla\psi \cdot d\mathbf n\right|\\
&\lesssim \|\psi\|_{W^{2,p}(D)}+\|\nabla\psi\|_{L^{p}(\partial D)}\\
&\lesssim \|\psi\|_{W^{2,p}(D)}.
\end{split}
\end{equation}
Note that the trace inequality  was used in the last inequality of \eqref{appd9}. Hence \eqref{appd1} has been proved.

We next prove \eqref{appd2}. The second inequality is obvious. To prove the first inequality, notice that
\[\|\omega\|_{L^{p}(D)}\lesssim \|\mathbf v\|_{W^{1,p}(D)},\]
\[|a_{i}|= \left|\int_{\Gamma_{i}}\nabla\psi\cdot d\mathbf n\right|=\left|\int_{\Gamma_{i}}\mathbf v\cdot d\mathbf r\right|\lesssim \|\mathbf v\|_{L^p(\partial D)}\lesssim  \|\mathbf v\|_{W^{1,p}(D)}, \]
which in combination with the first inequality in \eqref{appd1} yields
\[
 \|\psi\|_{W^{2,p}(D)}\lesssim \|\mathbf v\|_{W^{1,p}(D)}.\]
 Hence \eqref{appd2} is proved.
\end{proof}


The following strong maximum principle for superharmonic functions will be used in the proof of Proposition \ref{keyprop}.
\begin{lemma}[\cite{LL}, p. 246]\label{smple}
Suppose $u\in H^2(D)$ (thus $u\in C(\bar D)$ by the Sobolev embedding theorem) is nonnegative, and satisfies
 \[-\Delta u\geq 0 \,\,\mbox{ a.e. in }D.\]
Then either $u\equiv 0$ in $D$, or else $u>0$ in $D$.

\end{lemma}

\section{Constrained variational problems}\label{section3}
In this section, we study several constrained variational problems in the function spaces $\mathcal X$ and $\mathcal Y$ (defined by \eqref{defx},  \eqref{defy} in Section 1), which are involved at many places in this paper.

We begin with the following constrained minimization problem in $\mathcal X$:
\begin{equation}\label{var1}
\lambda_{c}=\inf\left\{\int_{D}|\nabla u|^{2}+cu^{2}dx\mid u\in\mathcal X,\,\,\|u\|_{L^{2}(D)}=1\right\}.
\end{equation}

If $c\equiv 0$, we denote the minimum by $\lambda$, i.e.,
\begin{equation}\label{var3}
\lambda=\inf\left\{\int_{D}|\nabla u|^{2}dx\mid u\in\mathcal X, \,\,\|u\|_{L^{2}(D)}=1\right\}.
\end{equation}

\begin{proposition}\label{prop1}
Given $c\in L^\infty(D)$, the following assertions hold:
\begin{itemize}
\item [(i)] There exists a minimizer for \eqref{var1}.
\item [(ii)] If $\tilde u$ is a minimizer of \eqref{var1}, then $\tilde u\in\mathcal Y$, and
\begin{equation}\label{rop1}
-\Delta\tilde u+c\tilde u=\lambda_{c}\tilde u\quad \mbox{ a.e. in }D,
\end{equation}
\begin{equation}\label{rop2}
\int_{D}\nabla \tilde u\cdot\nabla\zeta_{i}dx+\int_{D}c\tilde u\zeta_{i}dx=\lambda_{c}\int_{D}\tilde u\zeta_{i}dx,\quad i=1,\cdot\cdot\cdot,N.
\end{equation}
\item[(iii)]It holds that
\begin{align}
\lambda_{c}&=\inf\left\{\int_{D}|\nabla u|^{2}+cu^{2}dx\mid u\in\mathcal Y, \,\,\|u\|_{L^{2}(D)}=1\right\}\label{zj8}\\
&=\inf\left\{\int_{D}\phi\mathcal P\phi+c(\mathcal P\phi)^{2}dx\mid \phi\in L^2(D), \,\,\|\mathcal P\phi\|_{L^{2}(D)}=1\right\}.\label{zj811}
\end{align}
In particular, it holds that
\begin{align}
\lambda&=\inf\left\{\int_{D}|\nabla u|^{2}dx\mid u\in\mathcal Y, \,\,\|u\|_{L^{2}(D)}=1\right\}\label{kzj8}\\
&=\inf\left\{\int_{D}\phi\mathcal P\phi dx\mid \phi\in L^2(D), \,\,\|\mathcal P\phi\|_{L^{2}(D)}=1\right\}.\label{kzj811}
\end{align}
\item[(iv)] Let $\tilde u$ be a minimizer of \eqref{var1}. Denote $\tilde u^+=\max\{\tilde u,0\}$, $\tilde u^-=\max\{-\tilde u,0\}.$ Then $\tilde u^+, \tilde u^-\in\mathcal Y$, and
\begin{equation}\label{rop1o}
-\Delta\tilde u^++c\tilde u^+=\lambda_{c}\tilde u^+\quad \mbox{a.e. in }D,
\end{equation}
\begin{equation}\label{rop1p}
-\Delta\tilde u^-+c\tilde u^-=\lambda_{c}\tilde u^-\quad \mbox{a.e. in }D.
\end{equation}
\end{itemize}
\end{proposition}




\begin{proof}
We first prove (i). For any $u\in \mathcal X,$ $\|u\|_{L^{2}(D)}=1,$ we have that
\[-\|c\|_{L^\infty(D)} \leq \int_{D}|\nabla u|^{2}+cu^{2}dx<+\infty. \]
Hence $\lambda_{c}\in\mathbb R.$ Choose a sequence $\{u_{n}\}\subset\mathcal X$, $\|u\|_{L^{2}(D)}=1$ for every $n$, such that
\begin{equation}\label{rop30}
\int_{D}|\nabla u_{n}|^{2}+cu_{n}^{2}dx\to\lambda_{c} \quad\mbox{as }n\to+\infty.
\end{equation}
Obviously $\{u_{n}\}$ is bounded in $H^{1}(D)$ and thus, up to a subsequence, it has a weak limit, say $\tilde u$, in $H^{1}(D)$. By Proposition \ref{propb1}(ii),  $\tilde u\in \mathcal X.$  In addition, since $H^{1}(D)$ is compactly embedded in $L^{2}(D),$ we infer that $\|\tilde u\|_{L^{2}(D)}=1.$
Hence, by the definition of $\lambda_c$,
\begin{equation}\label{kja1}
\int_{D}|\nabla \tilde u|^{2}+c\tilde u^{2}dx\geq \lambda_c.
\end{equation}
On the other hand,  by weak lower semicontinuity,
\begin{equation}\label{kja2}
\int_{D}|\nabla \tilde u|^{2}+c\tilde u^{2}dx\leq \liminf_{n\to+\infty}\int_{D}|\nabla   u_n|^{2}+c u_n^{2}dx=\lambda_{c}.
\end{equation}
From \eqref{kja1} and \eqref{kja2}, we see that $\tilde u$ is a maximzier of \eqref{var1}.

Next we prove (ii). Let $\tilde u$ be a maximizer of \eqref{var1}. For any $u\in \mathcal X$, define
\[u_{\varepsilon}=\frac{\tilde u+\varepsilon u}{\|\tilde u+\varepsilon u\|_{L^{2}(D)}},\]
where $\varepsilon\in\mathbb R$ is small in absolute value (such that $\|\tilde u+\varepsilon u\|_{L^{2}(D)}>0$).  Then we have
\begin{equation}\label{daos0}
\frac{d}{d\varepsilon}\left(\int_{D}|\nabla u_{\varepsilon}|^{2}+cu_{\varepsilon}^{2}dx\right)\bigg|_{\varepsilon=0}=0.
\end{equation}
After straightforward computations, we obtain from \eqref{daos0} that
\begin{equation}\label{daos1}
\int_{D}\nabla \tilde u\cdot\nabla u+c\tilde u udx=\lambda_{c}\int_{D}\tilde uudx.
\end{equation}
Write $u=v+\sum_{i=1}^{N}\theta_{i}\zeta_{i},$ where $ v\in H^{1}_{0}(D)$, $\theta_1,\cdot\cdot\cdot,\theta_{N}\in\mathbb R$. Then \eqref{daos1} becomes
\begin{equation}\label{daos2}
\begin{split}
&\int_{D}\nabla \tilde u\cdot\nabla v+c\tilde u vdx +\sum_{i=1}^{N}\left(\theta_{i}\int_{D}\nabla \tilde u\cdot\nabla \zeta_{i}+c\tilde u\zeta_{i}dx\right)\\
=&\lambda_{c}\int_{D}\tilde uvdx+\lambda_{c}\sum_{i=1}^{N}\left(\theta_{i}\int_{D}  u\zeta_{i}dx\right).
\end{split}
\end{equation}
Since \eqref{daos1} holds for arbitrary $u\in\mathcal X$, we see that \eqref{daos2} holds for any $v\in H^{1}_{0}(D)$ and $\theta_{1},\cdot\cdot\cdot,\theta_{N}\in \mathbb R,$ which results in
\begin{equation}\label{daos3}
\int_{D}\nabla \tilde u\cdot\nabla v+c\tilde u vdx  =\lambda_{c}\int_{D}\tilde uvdx,\quad\forall\,v\in H^{1}_{0}(D).
\end{equation}
\begin{equation}\label{daos4}
 \int_{D}\nabla \tilde u\cdot\nabla \zeta_{i}+c\tilde u\zeta_{i}dx= \lambda_{c}\int_{D}  \tilde u\zeta_{i}dx,\quad i=1,\cdot\cdot\cdot,N.
\end{equation}
By standard elliptic regularity theory we have that $\tilde u \in H^2(D)$, and  \begin{equation}\label{daos5}
-\Delta \tilde u+c\tilde u=\lambda_{c}\tilde u\quad\mbox{a.e. in }D.
\end{equation}
To complete the proof of (ii), it suffices to verify
\begin{equation}\label{daos6}
\int_{\Gamma_{i}}\nabla\tilde u\cdot d\mathbf n=0,\quad i=1,\cdot\cdot\cdot,N.
\end{equation}
In fact, using the Stokes theorem, we have that
\begin{equation*}
\int_{\Gamma_{i}}\nabla\tilde u\cdot d\mathbf n=\int_{\partial D}\zeta_{i}\nabla\tilde u\cdot d\mathbf n=\int_{ D}\nabla \zeta_{i}\cdot\nabla\tilde u +\zeta_{i}\Delta\tilde u dx=\int_{ D}\nabla \zeta_{i}\cdot\nabla\tilde u +(c-\lambda_{c})\tilde u\zeta_{i} dx=0.
\end{equation*}
Here we used \eqref{daos4} and \eqref{daos5}.

Next we prove (iii). Since $\mathcal Y\subset\mathcal X,$ we have
\[ \lambda_{c}\leq \inf\left\{\int_{D}|\nabla u|^{2}+cu^{2}dx\mid u\in\mathcal Y, \,\,\|u\|_{L^{2}(D)}=1\right\}.\]
On the other hand, since any maximizer $\tilde u$ of \eqref{var1} belongs to $\mathcal Y$ (by item (ii)), we have that
\[ \inf\left\{\int_{D}|\nabla u|^{2}+cu^{2}dx\mid u\in\mathcal Y, \,\,\|u\|_{L^{2}(D)}=1\right\}\leq \int_{D}|\nabla \tilde u|^{2}+c\tilde u^{2}dx=\lambda_{c}.\]
Thus  \eqref{zj8} is proved.  In combination with Proposition \ref{propb2}(ii)(iii),  we get \eqref{zj811}.

We finally prove (iv).  Let $\tilde u$ be a maximizer of \eqref{var1}. It is clear that $\tilde u^{+}, \tilde u^{-}\in\mathcal X.$  Hence, by the definition of $\lambda_{c}$,
\begin{equation}\label{daos10}
\int_{D}|\nabla \tilde u^{+}|^{2}+c(\tilde u^{+})^{2}dx\geq \lambda_{c}\int_{D}(\tilde u^{+})^{2}dx,
\end{equation}
\begin{equation}\label{daos11}
 \int_{D}|\nabla \tilde u^{-}|^{2}+c(\tilde u^{-})^{2}dx\geq \lambda_{c}\int_{D}  (\tilde u^{-})^{2}dx.
\end{equation}
On the other hand,  we have that
\begin{equation}\label{daos12}
\lambda_{c}=\int_{D}|\nabla \tilde u |^{2}+c\tilde u^{2}dx=\int_{D}|\nabla \tilde u^{+}|^{2}+c(\tilde u^{+})^{2}dx+\int_{D}|\nabla \tilde u^{-}|^{2}+c(\tilde u^{-})^{2}dx,
\end{equation}
\begin{equation}\label{daos13}
1=\int_{D}\tilde u^{2}dx=\int_{D}  (\tilde u^{+})^{2}dx+\int_{D}  (\tilde u^{-})^{2}dx.
\end{equation}
Combining \eqref{daos10}-\eqref{daos13}, we deduce that
\begin{equation}\label{daos14}
\int_{D}|\nabla \tilde u^{+}|^{2}+c(\tilde u^{+})^{2}dx= \lambda_{c}\int_{D}(\tilde u^{+})^{2}dx,
\end{equation}
\begin{equation}\label{daos141}
 \int_{D}|\nabla \tilde u^{-}|^{2}+c(\tilde u^{-})^{2}dx= \lambda_{c}\int_{D}  (\tilde u^{-})^{2}dx.
\end{equation}
To verify \eqref{rop1o},
we assume, without loss of generality,  that $\tilde u^{+}\neq 0.$ Then \eqref{daos14} implies that $\tilde u^{+}/\|\tilde u^{+}\|_{L^{2}(D)}$ is a maximizer of \eqref{var1}. Hence, by (ii),  $\tilde u^{+}\in \mathcal Y$ and satisfies
\[-\Delta \tilde u^{+}+c\tilde u^{+}=\lambda_{c}\tilde u^{+}\quad\mbox{a.e. in }D.\]
The proof of \eqref{rop1p} is similar.
\end{proof}

As a consequence of Proposition \ref{prop1}(iii), we have the following obvious corollary, which will be used in the proof of Proposition \ref{keyprop}.

 \begin{corollary}\label{glav2022}
 Let $c\in L^\infty(D)$. If
\begin{equation}\label{glav92}
 \int_{D}|\nabla u|^{2}+cu^{2}dx\geq 0,\quad \forall\, u\in \mathcal  Y,
 \end{equation}
 then
 \begin{equation}\label{glav91}
 \int_{D}|\nabla u|^{2}+cu^{2}dx\geq 0,\quad \forall\, u\in \mathcal  X.
 \end{equation}
\end{corollary}

We are also interested in the following maximization problem:
\begin{equation}\label{var2}
\Lambda=\sup \left\{\int_{D}\phi\mathcal P\phi dx\mid \phi\in L^{2}(D), \,\,\|\phi\|_{L^{2}(D)}=1\right\}.
\end{equation}
It is easy to see that $\Lambda$ is a positive number. In addition, by Proposition \ref{propb2},  we have
\begin{equation}\label{var10}
\Lambda=\sup \left\{\int_{D}|\nabla u|^{2}dx\mid u\in\mathcal Y, \|\Delta u\|_{L^{2}(D)}=1\right\}.
\end{equation}

\begin{proposition}\label{jjzuo}
  There exists a maximizer for \eqref{var2}. Moreover, any maximizer $\tilde \phi$  of \eqref{var2} satisfies
\begin{equation}\label{kja5}
\mathcal P\tilde\phi=\Lambda\tilde\phi.
\end{equation}
\end{proposition}
\begin{proof}
Choose a sequence $\{\phi_{n}\}\subset L^{2}(D)$, $\|\phi_n\|_{L^{2}(D)}=1$ for every $n$,  such that
\[\int_{D}\phi_{n}\mathcal P\phi_{n}dx\to\Lambda \quad\mbox{as }n\to+\infty.\]
We assume, up to a subsequence, that $\phi_{n}$ converges weakly to $\tilde \phi$ in $L^{2}(D)$. By weak lower semicontinuity, we have
$\|\tilde\phi\|_{L^{2}(D)}\leq 1$.
Besides, by Proposition \ref{eiyoyo}(i) we have that
 $\mathcal P\phi_{n}\to \mathcal P\tilde \phi$ in $L^{2}(D)$. Therefore
\begin{equation}\label{kja10}
\int_{D}\tilde \phi\mathcal P\tilde\phi dx=\lim_{n\to+\infty}\int_{D}\phi_{n}\mathcal P\phi_{n}dx=\Lambda.
\end{equation}
In order to show that $\tilde \phi$ is a maximizer, it is enough to prove that $\|\tilde \phi\|_{L^{2}(D)}=1.$ Suppose by contradiction that $\|\tilde \phi\|_{L^{2}(D)}<1$. Since  $\tilde \phi\neq 0$ (due to \eqref{kja10}), we can compute
\[\int_{D}\left(\frac{\tilde \phi}{\|\tilde \phi\|_{L^{2}(D)}}\right)\mathcal P\left(\frac{\tilde \phi}{\|\tilde \phi\|_{L^{2}(D)}}\right)=\frac{\Lambda}{\|\tilde \phi\|_{L^{2}(D)}^{2}}>\Lambda,\]
a contradiction to the definition of $\Lambda.$  Hence the existence of a maximizer is proved.

Below we show that any maximizer $\tilde \phi$ satisfies \eqref{kja5}.
For any  $\phi\in L^{2}(D)$, define
\[\phi_{\varepsilon}=\frac{\tilde \phi+\varepsilon \phi}{\|\tilde \phi+\varepsilon \phi\|_{L^{2}(D)}},\]
where $\varepsilon\in\mathbb R$ is small in absolute value such that  $\|\tilde \phi+\varepsilon \phi\|_{L^{2}(D)}>0$.  Then
\begin{equation}\label{varr1}
\frac{d}{d\varepsilon}\left(\int_{D}\phi_{\varepsilon}\mathcal P\phi_{\varepsilon}dx\right)\bigg|_{\varepsilon=0}=0,
\end{equation}
which implies after some simple computations that
\begin{equation}\label{varr2}
\int_{D}\phi\mathcal P\tilde\phi dx=\Lambda\int_{D}\phi \tilde\phi dx.
\end{equation}
Since \eqref{varr2} holds for any  $\phi\in L^{2}(D)$, \eqref{kja5} is proved.
\end{proof}

The following proposition reveals the relations between the two variational problems \eqref{var3} and \eqref{var2}.

 \begin{proposition}\label{prop37}
The two variational problems \eqref{var3} and \eqref{var2} are equivalent, in the sense that:
 \begin{itemize}
\item[(i)]  $\lambda\Lambda=1$;
\item[(ii)] if $\tilde \phi$ is a maximizer of \eqref{var2}, then $\Lambda^{-1}(\mathcal P\tilde\phi)$  is a minimizer of \eqref{var3};
\item[(iii)] if $\tilde u$ is a maximizer of \eqref{var3}, then  $\lambda^{-1}(-\Delta\tilde u)$ is a minimizer of \eqref{var2}.
\end{itemize}
\end{proposition}
\begin{proof}
Let $\tilde \phi$ be a maximizer for \eqref{var2}. Then  $
\tilde\phi$ satisfies
\[\| \tilde\phi\|_{L^{2}(D)}=1,\quad \int_{D}\tilde\phi \mathcal P\tilde\phi dx=\Lambda.\]
By \eqref{kja5}, we have that
\[\|\mathcal P\tilde\phi\|_{L^2(D)}=\Lambda\|\tilde\phi\|_{L^2(D)}=\Lambda.\]
Taking into account \eqref{zj811} (taking $c\equiv0$ therein),
\begin{equation}\label{varr20}
\lambda\leq \int_{D}\left(\frac{\tilde \phi}{\|\mathcal P\tilde\phi\|_{L^{2}(D)}}\right)\mathcal P\left(\frac{\tilde \phi}{\|\mathcal P\tilde\phi\|_{L^{2}(D)}}\right)dx=\frac{1}{\Lambda}.
\end{equation}
Thus  $\lambda\Lambda\le1.$

Now we prove the converse inequality. Let $\tilde u$ be a minimizer of \eqref{var3}. Then $\tilde u$ satisfies
\[\tilde u\in   \mathcal Y,\quad \|\tilde u\|_{L^2(D)}=1,\quad \int_{D}|\nabla \tilde u|^{2}dx=\lambda.\]
By Proposition \ref{prop1}(ii) (taking $c\equiv 0$ therein),
\[\|\Delta\tilde u\|_{L^{2}(D)}=\lambda\|\tilde u\|_{L^2(D)}=\lambda.\]
Recalling \eqref{var10}, we have that
\begin{equation}\label{varr22}
\Lambda\ge \int_{D}\left|\nabla \left(\frac{\tilde u}{\|\Delta\tilde u\|_{L^{2}(D)}}\right)\right|^{2}dx=\frac{1}{\lambda},
\end{equation}
 which yields $\lambda\Lambda\ge1.$ Hence (i) is proved.

Since $\lambda\Lambda=1$, we see that  \eqref{varr20} and \eqref{varr22} are in fact equalities, from which we conclude (ii) and (iii) immediately.
\end{proof}

\begin{remark}
It is not hard to verify that $\lambda$ is strictly less than the  first eigenvalue of $-\Delta$ with zero boundary condition in $H^{1}_{0}(D),$ and any
  maximizer of \eqref{var3}  cannot belong to $H^{1}_{0}(D)$. We omit the proof here since they are not directly related to this paper, .
 \end{remark}


\section{Variational characterization}\label{varcha}

As the first step towards proving Theorem \ref{thm1}, we give a variational characterization for the steady flow in Theorem \ref{thm1} in terms of  conserved quantities of the two-dimensional Euler equations. More specifically, we will show that  the  vorticity of the steady flow under consideration is  an isolated local maximizer of the kinetic energy relative to its rearrangement class.

Throughout this section, let  $1<p<+\infty$ be fixed.  Let $\bar\psi$ be as in Theorem \ref{thm1}. Let  $\mathbf a=(a_1,\cdot\cdot\cdot,a_{N})$ be the circulation vector of the steady flow determined by $\bar\psi$, i.e.,
\[a_{i}=-\int_{\Gamma_{i}}\nabla\bar\psi\cdot d\mathbf n,\quad i=1,\cdot\cdot\cdot,N.\]
Let $\bar\omega=-\Delta\bar\psi$  be the corresponding vorticity. By Proposition \ref{propc2}, we have $\bar\psi=\mathcal P\bar\omega+h_{\mathbf a}.$ Hence $\bar\omega$ satisfies
\begin{equation}\label{glav5}
\bar\omega=g(\mathcal P\bar\omega+h_{\mathbf a}).
\end{equation}

Our main result in this section can be stated as follows.

\begin{theorem}\label{thm2}
 $\bar\omega$ is an isolated local maximizer of $E(\cdot,\mathbf a)$ relative to $\mathcal R_{\bar\omega},$ i.e.,  there exists some $r_{0}>0$ such that
 \begin{equation}\label{aspthm2}
  E(w,\mathbf a) < E(\bar\omega,\mathbf a),\quad\forall\,w\in \mathcal R_{\bar\omega},\,\,0<\|w-\bar\omega\|_{L^{p}(D)}<r_{0}.
 \end{equation}
 \end{theorem}

The proof of Theorem \ref{thm2} is based on the  energy-Casimir ($EC$)  method   and the  supporting  functional  method.

\subsection{$EC$ method and  supporting  functional method}\label{subsec51}

Before proving Theorem \ref{thm2}, we briefly review the energy-Casimir  method  by Arnold in the 1960s, and its further development, the  supporting  functional  method, by Wolansky and Ghil in the 1990s, in the study of nonlinear stability of two-dimensional steady Euler flows.

 To explain these two methods, we use them to prove the following special case of Theorem \ref{thm2} as an example.

\begin{theorem}[A special case of Theorem \ref{thm2}]\label{thm20}
In the setting of  Theorem \ref{thm1}, if \eqref{ssmmy} is replaced by the following  stronger condition:
 \begin{equation}\label{ssmmy2}
 \int_{D}|\nabla u|^{2}-g'(\bar\psi)u^{2}dx\geq \delta_{0}\int_{D}u^{2}dx,\quad \forall\, u\in \mathcal Y
 \end{equation}
for some $\delta_{0}>0$.
Then  $\bar\omega$ is an isolated local maximizer of $E(\cdot,\mathbf a)$ relative to $\mathcal R_{\bar\omega}.$
\end{theorem}

\begin{remark}
 By Lemma \ref{propb2}(ii)(iii), the condition \eqref{ssmmy2} can also be written as
  \begin{equation}\label{ssmmy3}
 \int_{D}\phi\mathcal P\phi dx-\int_{D} g'(\bar\psi)(\mathcal P\phi)^{2}dx\geq \delta_{0}\int_{D}(\mathcal P\phi)^{2}dx,\quad \forall\, \phi\in L^{2}(D).
 \end{equation}
 \end{remark}

The strategy of proving Theorem \ref{thm20} is as follows. First we define a functional  $EC$, which is equal to $E(\cdot,\mathbf a)$ plus some constant on $\mathcal R_{\bar\omega}$. Thus the problem can be reduced to proving that  $\bar\omega$ is an isolated local maximizer of $EC$ relative to $\mathcal R_{\bar\omega}.$ Then we construct a supporting functional $\mathcal D$, which is above $EC$, but touches $EC$ at $\bar\omega$.  Finally we show that  $\bar\omega$ is an isolated local maximizer of  $\mathcal D$ under the condition \eqref{ssmmy2}.

In order to define $EC$ and the supporting functional, we need some preparations.  Denote
\begin{equation}\label{zbj11}
\bar m=\min_{\bar D}\bar\psi,\quad \bar M=\max_{\bar D}\bar\psi.
\end{equation}
Since any change of $g$ outside the interval $[\bar m,\bar M]$ makes no difference to the validity of Theorem \ref{thm20},  we can redefine $g$ outside  $[\bar m,\bar M]$ such that
\begin{equation}\label{gco1}
g\in C^{1}(\mathbb R),
\end{equation}
\begin{equation}\label{gco2}
\mbox{$ g$ is  increasing,}
\end{equation}
\begin{equation}\label{gco3}
\mbox{$\lim_{s\to+\infty}\frac{ g(s)}{s}=c_1,\quad \lim_{s\to-\infty}\frac{ g(s)}{s}=c_2$ for some positive constants $c_1,c_2$.}
\end{equation}
This is doable by Lemma \ref{mM}.
 Denote
\[G(s)=\int_{0}^{s}g(\tau)d\tau,\quad s\in\mathbb R.\]
Let $\hat G$ be the Legendre transform of $G$. By Lemma \ref{lt}(ii), $\hat G$ is local Lipschitz continuous on $\mathbb R$.

 Define the energy-Casimir functional
 \begin{equation}\label{glav1}
 EC(w)=E(w,\mathbf a)-\int_D\hat G(w)dx,\quad w\in \mathcal R_{\bar\omega}.
 \end{equation}
 Note that $\mathbf a$ is fixed through this section, so we can avoid the longer notation $EC(w,\mathbf a).$  In the definition of $EC$, the first term is the kinetic energy of the fluid, and the second term is called the Casimir functional.

 Define
  \begin{equation}\label{glae70}
  D(w)=-\frac{1}{2}\int_{D}w\mathcal Pwdx+\int_{D}G(\mathcal Pw+h_{\mathbf a})+\frac{1}{2}\sum_{i,j=1}^Nq_{ij}a_ia_j.
  \end{equation}

\begin{lemma}\label{ymys00}
$\mathcal D$ is a supporting functional of $EC$ at $\bar\omega,$ i.e.,
\begin{itemize}
\item[(i)] $\mathcal D(w)\geq EC(w)$ for any $w\in\mathcal R_{\bar\omega};$
\item[(ii)]$\mathcal D(\bar\omega)=EC(\bar\omega).$
\end{itemize}
\end{lemma}

\begin{proof}
Recalling the definition of $E(w,\mathbf a)$ (see \eqref{kevv0}), we have
\[EC(w)=\frac{1}{2}\int_{D}w\mathcal Pwdx+\int_{D}h_{\mathbf a}wdx+\frac{1}{2}\sum_{i,j=1}^Nq_{ij}a_ia_j-\int_{D}\hat G(w)dx,\]
By Lemma \ref{lt}(i), we have that
 \begin{equation}\label{ymys0}
\begin{split}
EC(w)&=-\frac{1}{2}\int_{D}w\mathcal Pwdx+\int_{D}(\mathcal Pw+h_{\mathbf a})w-\hat G(w)dx+\frac{1}{2}\sum_{i,j=1}^Nq_{ij}a_ia_j\\
&\leq -\frac{1}{2}\int_{D}w\mathcal Pwdx+\int_{D}G(\mathcal Pw+h_{\mathbf a})dx+\frac{1}{2}\sum_{i,j=1}^Nq_{ij}a_ia_j\\
&=\mathcal D(w).
\end{split}
\end{equation}
To prove (ii), notice that
\begin{equation}\label{ymys1}
\hat G(\bar\omega)+G(\mathcal P\bar\omega+h_{\mathbf a})=\bar\omega(\mathcal P\bar\omega+h_{\mathbf a}).
\end{equation}
Here we used  Lemma \ref{lt}(i) and the fact that $\bar\omega=g(\mathcal P\bar\omega+h_{\mathbf a}).$
Therefore the inequality in \eqref{ymys0} is in fact an equality if we take $w=\bar\omega$.
\end{proof}

\begin{proof}[Proof of Theorem \ref{thm20}]
Since the Casimir functional is constant on $\mathcal R_{\bar\omega},$ we only need to show that $\bar\omega$ is an isolated local maximizer of $EC$
relative to $\mathcal R_{\bar\omega}$. By Lemma \ref{ymys00},  this can be reduced to proving that $\bar\omega$ is an isolated local maximizer of $\mathcal D$ relative to $\mathcal R_{\bar\omega}.$

Suppose by contradiction that  $\bar\omega$ is not an isolated local maximizer of $\mathcal D$ on $\mathcal R_{\bar\omega}.$   Then there exists a sequence
$\{w_{n}\}\subset\mathcal R_{\bar\omega}$ such that $w_{n}\neq \bar\omega$ for every $n,$ $ w_n$ converges to $\bar\omega$ in $L^{p}(D)$ as  $n\to+\infty,$ and ${\mathcal D}(w_{n})\geq  {\mathcal D}(\bar\omega)$ for every $n$.
For simplicity, denote $\phi_{n}=w_{n}-\bar\omega.$ Then the sequence $\{\phi_n\}$  satisfies
\begin{equation}\label{phi0p00}
\phi_{n}\neq 0,\quad\forall\,n,
\end{equation}
\begin{equation}\label{phi0p}
 \lim_{n\to+\infty}\|\phi_{n}\|_{L^{p}(D)}=0,
\end{equation}
\begin{equation}\label{dayu1}
 {\mathcal D}(\bar\omega+\phi_n)\geq  {\mathcal D}(\bar\omega),\quad \forall\,n.
 \end{equation}
By \eqref{dayu1}, we have that
\begin{equation}\label{glav10}
\begin{split}
-\frac{1}{2}\int_{D}(\bar\omega+\phi_{n})\mathcal P(\bar\omega+\phi_{n})dx&+\int_{D}G(\mathcal P\bar\omega+\mathcal P\phi_{n}+h_{\mathbf a})dx\\
&\geq -\frac{1}{2}\int_{D}\bar\omega \mathcal P\bar\omega dx+\int_{D}G(\mathcal P\bar\omega +h_{\mathbf a})dx,\quad \forall\,n.
\end{split}
\end{equation}
Simplifying \eqref{glav10} gives
\begin{equation}\label{jifenyx0}
\int_{D}\bar\omega\mathcal P\phi_{n}dx+\frac{1}{2}\int_{D}\phi_{n}\mathcal P\phi_{n}dx\leq \int_{D}G(\mathcal P\bar\omega+\mathcal P\phi_{n}+h_{\mathbf a})-G(\mathcal P\bar\omega+h_{\mathbf a})dx,\quad\forall\,n.
\end{equation}
Here we used the symmetry of $\mathcal P$ (see Lemma \ref{eiyoyo}(ii)).
To proceed, we recall Taylor's theorem, i.e.,
\begin{equation}\label{jifenyx01}
G(s+s_{0})=G(s_{0})+g(s_{0})s+\frac{1}{2}g'(s_{0})s^{2}+\mathfrak r(s)s^{2},\quad \forall\,s_{0}, s\in\mathbb R,
\end{equation}
where $\mathfrak r\in C(\mathbb R)$ satisfying   $\mathfrak r(s)\to$ 0 as $s\to0$.
By \eqref{jifenyx01}, and taking into account \eqref{glav5}, the integrand on the right-hand side of \eqref{jifenyx0} can be written as
\begin{equation}\label{jifenyx1}
G(\mathcal P\bar\omega+\mathcal P\phi_{n}+h_{\mathbf a})-G(\mathcal P\bar\omega+h_{\mathbf a})=\bar\omega\mathcal P\phi_{n}+\frac{1}{2}g'(\bar\psi)(\mathcal P\phi_{n})^{2}+\mathfrak r(\mathcal P\phi_{n})(\mathcal P\phi_{n})^{2}.
\end{equation}
Inserting \eqref{jifenyx1} into \eqref{jifenyx0} gives
\begin{equation}\label{jifenyx2}
 \frac{1}{2}\int_{D}\phi_{n}\mathcal P\phi_{n}dx\leq \frac{1}{2}\int_{D}g'(\bar\psi)(\mathcal P\phi_{n})^{2}dx+\int_{D}\mathfrak r(\mathcal P\phi_{n})(\mathcal P\phi_{n})^{2}dx,\quad\forall\,n.
\end{equation}
In combination with \eqref{ssmmy3}, we obtain
\begin{equation}\label{jifenyx3}
 \frac{1}{2}\delta_{0}\int_{D} (\mathcal P\phi_{n})^{2}dx\leq \int_{D}\mathfrak r(\mathcal P\phi_{n})(\mathcal P\phi_{n})^{2}dx,\quad\forall\,n.
\end{equation}

Below we deduce a contradiction.
In view of \eqref{phi0p}, we get $\mathcal P\phi_{n}\to0$ in $C(\bar D)$ by Lemma \ref{eiyoyo}(i) and the Sobolev embedding theorem. Since $\mathfrak r\in C(\mathbb R),$ we further infer that $\mathfrak r(\mathcal P\phi_{n})\to0$ in $C(\bar D)$. Therefore  \eqref{jifenyx3} is impossible since $\mathcal P\phi_{n}\neq 0$ for any $n$ by \eqref{phi0p00}.

\end{proof}

\subsection{Proof of variational characterization}

In this section, we give the proof of Theorem \ref{thm2}.
To begin with, observe that if $\int_{D}g'(\bar \psi)dx=0$, then $g'(\bar\psi)\equiv0$ in $D$ (since $g$ is increasing), which implies that
\[\nabla\bar\omega=\nabla(g(\bar\psi))=g'(\bar\psi)\nabla\bar\psi\equiv 0 \quad\mbox{in }D.\]
Hence in this case $\bar\omega$ must be constant in $D$, and thus the conclusion of Theorem \ref{thm2} is trivial.
For this reason, below we assume that
\begin{equation}\label{gpied0}
\int_{D}g'(\bar \psi)dx>0.
\end{equation}

The first step of the proof is  show that the nonnegative condition \eqref{ssmmy} actually implies certain weak positive definiteness.

\begin{proposition}[Weak positive definiteness]\label{keyprop}
Suppose \eqref{gpied0} holds. Then  there  exists some $\delta_0>0$ such that
\begin{equation}\label{lemc1}
\int_{D}|\nabla u|^2 dx-\int_{D}g'(\bar\psi)u^{2}dx+\frac{\left(\int_{D}g'(\bar\psi)udx\right)^{2}}{\int_{D}g'(\bar\psi)dx}\geq \delta_0\int_{D}u^{2}dx,\quad\forall\,u\in \mathcal Y.
\end{equation}
\end{proposition}
\begin{remark}
By Proposition \ref{propb2}(ii)(iii), \eqref{lemc1} is equivalent to
\begin{equation}\label{lemc1220}
\int_{D}\phi\mathcal P\phi dx-\int_{D}g'(\bar\psi)(\mathcal P\phi)^{2}dx+\frac{\left(\int_{D}g'(\bar\psi)\mathcal P\phi dx\right)^{2}}{\int_{D}g'(\bar\psi)dx}\geq \delta_0\int_{D}(\mathcal P\phi)^{2}dx,\quad\forall\,\phi\in L^2(D).
\end{equation}
\end{remark}

\begin{proof}
We prove \eqref{lemc1} by contradiction.
Assume that, for every positive integer $n$, there exists $u_{n}\in \mathcal Y$, $\|u_n\|_{L^2(D)}=1,$ such that
\begin{equation}\label{lemc2}
\int_{D}|\nabla u_n|^2 dx-\int_{D}g'(\bar\psi)u_n^{2}dx+\frac{\left(\int_{D}g'(\bar\psi)u_n dx\right)^{2}}{\int_{D}g'(\bar\psi)dx}<\frac{1}{n}.
\end{equation}

By the condition \eqref{ssmmy},
\begin{equation}\label{lemc5}
\int_{D}|\nabla u_n|^2 dx-\int_{D}g'(\bar\psi)u_n^{2}dx\geq 0,\quad\forall\,n.
\end{equation}
Hence, from \eqref{lemc2} and \eqref{lemc5},  we have  that
\begin{equation}\label{lemc6}
\lim_{n\to+\infty}\int_{D}|\nabla u_n|^2 dx-\int_{D}g'(\bar\psi)u_n^{2}dx=0,
\end{equation}
\begin{equation}\label{lemc7}
\lim_{n\to+\infty} \int_{D}g'(\bar\psi)u_{n} dx=0.
\end{equation}
From \eqref{lemc6}  and the assumption that $\|u_n\|_{L^2(D)}=1$ for every $n$, we deduce that  $\{u_{n}\}$ is  bounded  in $H^{1}(D)$. Hence, up to a subsequence, $\{u_n\}$ has a weak limit, say $\tilde u$, in $H^1(D).$ Since $H^1(D)$ is compactly embedded in $L^2(D)$, it holds that
\begin{equation}\label{glavd1}
\|\tilde u\|_{L^2(D)}=1.
\end{equation}
By \eqref{lemc6} and \eqref{lemc7}, we have that
\begin{equation}\label{lemc8}
\int_{D}|\nabla\tilde u|^{2} dx-\int_{D}g'(\bar\psi)\tilde u^{2}dx\leq \liminf_{n\to+\infty}\int_{D}|\nabla u_{n}|^{2} dx-\int_{D}g'(\bar\psi)u_n^{2}dx=0,
\end{equation}
\begin{equation}\label{lemc9}
 \int_{D}g'(\bar\psi)\tilde u dx=0.
\end{equation}
In addition, since $\{u_{n}\}\subset \mathcal Y\subset\mathcal X$ and $\mathcal X$ is weakly closed in $H^{1}(D)$ (see Proposition \ref{propb1}(ii)), we see that \begin{equation}\label{lemc666}
\tilde u\in\mathcal X.
\end{equation}

On the other hand, applying Corollary \ref{glav2022} (taking $c=-g'(\bar\psi)$ therein), and taking into account \eqref{ssmmy}, we obtain
\begin{equation}\label{lemc10}
\int_{D}|\nabla u|^{2} dx-\int_{D}g'(\bar\psi) u^{2}dx\geq 0,\quad\forall\,u\in\mathcal X.
\end{equation}
Therefore, by \eqref{lemc8}, \eqref{lemc666} and \eqref{lemc10}, we infer that $\tilde u$ is a minimizer of the following minimization problem
\[\inf\left\{\int_{D}|\nabla u|^{2}-g'(\bar\psi)u^{2}dx\mid u\in \mathcal X,\,\, \|u\|_{L^{2}(D)}=1\right\}\]
with
\begin{equation}\label{lemc11}
\int_{D}|\nabla\tilde u|^{2} dx-\int_{D}g'(\bar\psi)\tilde u^{2}dx=0,
\end{equation}
Since $\tilde u\not\equiv0$ (by \eqref{glavd1}), without loss  of generality, we assume that $\tilde u^+\not\equiv 0$ (the case $\tilde u^-\not\equiv 0$ is similar).
By Proposition \ref{prop1}(iv),  $\tilde u^+\in H^2(D)$  and satisfies
 \[-\Delta \tilde u^+=g'(\bar\psi)\tilde u^+\geq 0.\]
So $\tilde u^+$ is  nonnegative and superharmonic. Then applying Lemma \ref{smple} gives
\[\tilde u^+>0   \quad\mbox{in }D,\]
which implies
\[\tilde u>0  \quad \mbox{in }D.\]
This obviously contradicts \eqref{lemc9} since $g'(\bar\psi)\geq 0$ in $D$ and $g'(\bar\psi)\not\equiv 0.$

\end{proof}

With Proposition \ref{keyprop} at hand, we are ready to prove Theorem \ref{thm2}.
The basic idea  is very similar to the one used in the proof of Theorem \ref{thm20}, but the computations are more involved. First we construct a new supporting functional, which is better (in the sense that it is   ``closer" to $EC$) than the one in Section \ref{subsec51}. Then we show that weak positive definiteness obtained in Proposition \ref{keyprop} ensures that the new supporting functional attains an isolated local maximum at $\bar\omega$.

Let $\bar m,$ $\bar M$ be defined as in Section \ref{subsec51}. Redefine $g$ outside the interval  $[\bar m,\bar M]$ according to Lemma \ref{mM} such that
\begin{equation}\label{gco11}
g\in C^{1}(\mathbb R),
\end{equation}
\begin{equation}\label{gco22}
\mbox{$ g$ is strictly increasing in $(-\infty, \bar m]$ and $[\bar M,+\infty)$,}
\end{equation}
\begin{equation}\label{gco33}
\mbox{$\lim_{s\to+\infty}\frac{ g(s)}{s}=c_1,\quad \lim_{s\to-\infty}\frac{ g(s)}{s}=c_2$ for some positive constants $c_1,c_2$.}
\end{equation}
 Note that  we require  $g$ to be \emph{strictly} increasing in \eqref{gco22}. This is mainly used in the proof of Lemma \ref{lm57} below.

 Let $G, \hat G$ and $EC$ be defined as in Section \ref{subsec51}.
 Define
 \begin{equation}\label{defds}
 D_{s}(w)=-\frac{1}{2}\int_{D}w\mathcal Pwdx+\int_{D}G(\mathcal Pw+h_{\mathbf a}-s)+s\mathsf m+\frac{1}{2}\sum_{i,j=1}^Nq_{ij}a_ia_j,
 \end{equation}
 where $s\in\mathbb R$ is a parameter, and $\mathsf m$ is a fixed real number given by
  \begin{equation}\label{msf108}
  \mathsf m=\int_{D}\bar\omega dx.
  \end{equation}
Define
  \begin{equation}\label{defdht}
  \hat{\mathcal D}(w)=\inf_{s\in\mathbb R}\mathcal D_{s}(w).
  \end{equation}

\begin{lemma}\label{lm56}
For any $w\in\mathcal R_{\bar\omega},$ there exists some  $  \mu\in\mathbb R,$ depending on $w$, such that
$\hat{\mathcal D}(w)=\mathcal D_{ \mu}(w).$ Moreover, such  $\mu$ necessarily satisfies
\begin{equation}\label{mude2}
\int_{D}g(\mathcal Pw+h_{\mathbf a}- \mu)dx=\mathsf m.
\end{equation}
\end{lemma}

\begin{proof}
Fix $w\in\mathcal R_{\bar\omega}$. It is easy to check that $\mathcal D_{s}(w)$ is  continuously differentiable with respect to $s$. By the assumption \eqref{gco33},
\begin{equation*}
\liminf_{|s|\to+\infty}\frac{G(s)}{s^{2}}>0,
\end{equation*}
which implies that
\[\lim_{s\to+\infty}\mathcal D_{ s}(w)=+\infty,\quad \lim_{s\to-\infty}\mathcal D_{ s}(w)=+\infty.\]
Therefore, by the intermediate value theorem,  there exists some $\mu\in\mathbb R$ such that
\[\mathcal D_{\mu}(w)=\inf_{s\in\mathbb R}\mathcal D_{s}(w).\]
Besides, for any such $\mu$, it holds that
 \[\frac{d\mathcal D_{s}(w)}{ds}\bigg|_{s=\mu}=0,\]
which gives \eqref{mude2}.

\end{proof}

In Lemma \ref{lm56}, the map $w\mapsto\mu$ is not single-valued in general. However, if $w=\bar\omega,$ we have the following result.

\begin{lemma}\label{lm57}
If  $\mu\in\mathbb R$ satisfies
\begin{equation}\label{mude3}
\int_{D}g(\mathcal P\bar\omega+h_{\mathbf a}- \mu)dx=\mathsf m,
\end{equation}
then $\mu=0.$
\end{lemma}
\begin{proof}
Suppose by contradiction that $\mu\neq 0$. Without loss of generality, assume that $\mu>0$. Since $g$ is increasing, we have
\begin{equation}\label{mude38d}
 g(\mathcal P\bar\omega+h_{\mathbf a}- \mu) \leq  g(\mathcal P\bar\omega+h_{\mathbf a}) \quad \mbox{in }D.
 \end{equation}
On the other hand,   by the assumption \eqref{gco22}, we have that
\begin{equation*}\label{mude39d}
\min_{x\in\bar D}g(\mathcal P\bar\omega+h_{\mathbf a}- \mu)=\min_{x\in\bar D}g(\bar\psi- \mu)=g(\bar m-\mu)<g(\bar m)=\min_{x\in\bar D}g(\bar\psi)=\min_{x\in\bar D}g(\mathcal P\bar\omega+h_{\mathbf a}),
 \end{equation*}
which implies that
\begin{equation}\label{mude40d}
g(\mathcal P\bar\omega+h_{\mathbf a}- \mu)\not\equiv g(\mathcal P\bar\omega+h_{\mathbf a}).
\end{equation}
 Combining \eqref{mude38d} and \eqref{mude40d}, we infer that
 \begin{equation}\label{mude41d}
 \int_Dg(\mathcal P\bar\omega+h_{\mathbf a}- \mu)dx<\int_Dg(\mathcal P\bar\omega+h_{\mathbf a})dx=\mathsf m,
 \end{equation}
 a contradiction.

\end{proof}

\begin{lemma}\label{lm58}
$\hat{\mathcal D}$ is a supporting functional of $EC$ at $\bar\omega$,  i.e.,
\begin{itemize}
\item[(i)] $\hat{\mathcal D}(w)\geq EC(w)$ for any $w\in\mathcal R_{\bar\omega};$
\item[(ii)]$\hat{\mathcal D}(\bar\omega)=EC(\bar\omega).$
\end{itemize}
\end{lemma}

\begin{proof}
We first prove (i). Fix $w\in \mathcal R_{\bar\omega}.$ For any $s\in\mathbb R$, we have that
\begin{align}
EC(w)&=\frac{1}{2}\int_{D}w\mathcal Pwdx+\int_{D}h_{\mathbf a}wdx +\frac{1}{2}\sum_{i,j=1}^Nq_{ij}a_ia_j-\int_{D}\hat G(w)dx\\
&=\frac{1}{2}\int_{D}w\mathcal Pwdx+\int_{D}h_{\mathbf a}wdx+\frac{1}{2}\sum_{i,j=1}^Nq_{ij}a_ia_j -\int_{D}\hat G(w)+swdx+s\mathsf m \label{dygd}\\
&=-\frac{1}{2}\int_{D}w\mathcal Pwdx +\int_{D}(\mathcal Pw+h_{\mathbf a}-s)w-\hat G(w)dx+s\mathsf m+\frac{1}{2}\sum_{i,j=1}^Nq_{ij}a_ia_j\\
&\leq -\frac{1}{2}\int_{D}w\mathcal Pwdx +\int_{D} G(\mathcal Pw+h_{\mathbf a}-s)dx+s\mathsf m +\frac{1}{2}\sum_{i,j=1}^Nq_{ij}a_ia_j \label{dygbd}\\
&=\mathcal D_{s}(w).
\end{align}
Note that we used the fact that $\int_{D}wdx=\int_{D}\bar\omega dx=\mathsf m$ (since $w\in\mathcal R_{\bar\omega}$) in \eqref{dygd}, and Lemma \ref{lt}(i) in \eqref{dygbd}. Hence $EC(w)\le\mathcal D(w)$.

Next we prove (ii). By Lemma \ref{lm57} and the obvious fact that
$\mathcal D_0(\bar\omega)= \mathcal D (\bar\omega)$ (recall  \eqref{glae70}),  we have
\[\hat{\mathcal D}(\bar\omega)=\mathcal D (\bar\omega).\]
  Then the desired result follows from Lemma \ref{ymys00}(ii).

\end{proof}

Having made enough preparations, we are ready to complete the proof of Theorem \ref{thm2}.

\begin{proof}[Proof of Theorem \ref{thm2}]

By Lemma \ref{lm58}, it is sufficient to show that $\bar\omega$ is an isolated local minimizer of $\hat{\mathcal D}$ relative to $ \mathcal R_{\bar\omega}.$

By contradiction, assume that there exists a sequence $\{w_{n}\}\subset\mathcal R_{\bar\omega}$ such that
\begin{equation*}\label{pofthm21}
w_{n}\neq \bar\omega\quad\mbox{for all }n,\quad \|w_n-\bar\omega\|_{L^{p}(D)}\to 0 \quad\mbox{as }n\to+\infty,
\quad\hat{\mathcal D}(w_{n})\geq \hat{\mathcal D}(\bar\omega)\quad\mbox{for all }n.
\end{equation*}
For simplicity, denote $\phi_{n}=w_{n}-\bar\omega.$ Then $\{\phi_n\}$ satisfies
\begin{equation}\label{pofthm22}
\phi_{n}\neq 0\quad\mbox{for all }n,
\end{equation}
\begin{equation}\label{pofthm23}
 \|\phi_{n}\|_{L^{p}(D)}\to 0 \quad\mbox{as }n\to+\infty,
\end{equation}
\begin{equation}\label{pofthm24}
\hat{\mathcal D}(\bar\omega+\phi_{n})\geq \hat{\mathcal D}(\bar\omega)\quad\mbox{for all }n.
\end{equation}
In addition, by Lemma \ref{lm56},  there exists $\mu_{n}\in\mathbb R$  such that
\begin{equation}\label{pofthm20}
\hat{\mathcal D}(\bar\omega+\phi_{n})=\mathcal D_{\mu_{n}}(\bar\omega+\phi_{n}),
\end{equation}
with $\mu_{n}$ satisfying
\begin{equation}\label{mucon1}
\int_{D}g(\bar\psi+\mathcal P\phi_n- \mu_{n})dx=\mathsf m.
\end{equation}

From \eqref{pofthm24} and \eqref{pofthm20}, after some simple algebraic operations, we obtain
\begin{equation}\label{pofthm28}
\int_{D}\bar\omega\mathcal P\phi_{n} dx+\frac{1}{2}\int_{D}\phi_{n}\mathcal P\phi_{n}dx\leq \int_{D}G(\bar\psi+\mathcal P\phi_{n}-\mu_{n})-G(\bar\psi)dx+\mu_{n}\mathsf m.
\end{equation}
Here we  used the symmetry of the operator $\mathcal P$ (see Lemma \ref{eiyoyo}(ii)). As in \eqref{jifenyx1}, the integrand on the right-hand side of \eqref{pofthm28} can be written as
\begin{equation}\label{pofthm29}
\begin{split}
&G(\bar\psi+\mathcal P\phi_{n} -\mu_{n})-G(\bar\psi)\\
=&g(\bar\psi)(\mathcal P\phi_{n}-\mu_{n})+\frac{1}{2}g'(\bar\psi)(\mathcal P\phi_{n}-\mu_{n})^{2}
 +\mathfrak r(\mathcal P\phi_{n}-\mu_{n})(\mathcal P\phi_{n}-\mu_{n})^{2}\\
 =&\bar\omega(\mathcal P\phi_{n}-\mu_{n})+\frac{1}{2}g'(\bar\psi)(\mathcal P\phi_{n}-\mu_{n})^{2}
 +\mathfrak r(\mathcal P\phi_{n}-\mu_{n})(\mathcal P\phi_{n}-\mu_{n})^{2},
\end{split}
\end{equation}
where $\mathfrak r\in C(\mathbb R)$ satisfying $\mathfrak r(s)\to 0$ as $s\to 0.$ 
Inserting \eqref{pofthm29} into \eqref{pofthm28}, and taking into account \eqref{msf108}, we get
\begin{equation}\label{pofthm210}
 \frac{1}{2}\int_{D}\phi_{n}\mathcal P\phi_{n}dx\leq  \frac{1}{2}\int_{D}g'(\bar\psi)(\mathcal P\phi_{n}-\mu_{n})^{2}dx
 +\int_{D}\mathfrak r(\mathcal P\phi_{n}-\mu_{n})(\mathcal P\phi_{n}-\mu_{n})^{2}dx.
\end{equation}
 For convenience, we write \eqref{pofthm210} as follows
\begin{equation}\label{pofthm2101}
\begin{split}
 & \int_{D}\phi_{n}\mathcal P\phi_{n}dx- \int_{D}g'(\bar\psi)(\mathcal P\phi_{n})^{2}dx\\
 \leq &  \mu_{n}^{2}\int_{D}g'(\bar\psi)dx-2\mu_{n}\int_{D}g'(\bar\psi)\mathcal P\phi_{n}dx
 +\int_{D}\vartheta_{n}(\mathcal P\phi_{n}-\mu_{n})^{2}dx,
 \end{split}
\end{equation}
where
\begin{equation}\label{pofthm210t}
\vartheta_{n}:=2\mathfrak r(\mathcal P\phi_{n}-\mu_{n}).
\end{equation}

Below we deduce a contradiction from \eqref{pofthm2101}. To begin with, we need some estimates for $\mu_{n}$.  Recall that $\mu_n$ satisfies \eqref{mucon1}. We first show that
\begin{equation}\label{muto0}
\lim_{n\to+\infty}\mu_{n}=0.
\end{equation}
By \eqref{pofthm23}, Lemma \ref{eiyoyo}(i) and the Sobolev embedding $W^{2,p}(D)\hookrightarrow C(\bar D)$, we deduce that
\begin{equation}\label{wnbdd}
\lim_{n\to+\infty}\|\mathcal P\phi_n\|_{L^{\infty}(D)}=0.
\end{equation}
In particular,
\begin{equation}\label{tpq11}
\mbox{$\{\mathcal P\phi_{n}\}$ is bounded in $L^{\infty}(D)$.}
\end{equation}
By \eqref{tpq11} and the fact that $|g(s)|\to+\infty$ as $|s|\to+\infty$ (due to \eqref{gco33}),
we conclude from \eqref{mucon1} that   $\{\mu_{n}\}$ is a bounded sequence. Therefore, \emph{up to a subsequence},  $\{\mu_{n}\}$ has a limit, say  $\mu$.    Passing to the limit $n\to+\infty$ in  \eqref{mucon1} gives
\begin{equation}\label{mucon167}
 \int_{D}g(\bar\psi- \mu)dx=\mathsf m.
\end{equation}
Hence, by Lemma \ref{lm57}, $\mu$ must be $0$. Such argument in fact indicates that any subsequence of $\{\mu_n\}$ has a subsequence converging to $0$, which implies that $\mu_{n}\to0$ as $n\to+\infty$. Therefore \eqref{muto0} is proved.

Note that from \eqref{muto0} and \eqref{wnbdd} we  have the following estimate for $\vartheta_n$:
\begin{equation}\label{thto0}
\lim_{n\to+\infty}\|\vartheta_{n}\|_{L^{\infty}(D)}=0.
\end{equation}

Below we deduce a better estimate for $\mu_{n}$:
\begin{equation}\label{muco489}
 \mu_{n}=\frac{\int_{D}g'(\bar\psi)\mathcal P\phi_{n}dx}{\int_{D}g'(\bar\psi)dx} +\alpha_{n},
\end{equation}
where $\{\alpha_{n}\}\subset\mathbb R$ satisfying
\begin{equation}\label{muco48u}
\lim_{n\to+\infty}\frac{\alpha_{n}}{ \|\mathcal P\phi_{n}\|_{L^{2}(D)}}=0.
\end{equation}
 To this end, first observe that \eqref{mucon1} can be written as
 \begin{equation}\label{mucon109}
\int_{D}g(\bar\psi+\mathcal P\phi_{n} - \mu_{n})-g(\bar\psi)dx=0.
\end{equation}
Applying Taylor's theorem again, we have that
\begin{equation}\label{mucon110}
\begin{split}
g(\bar\psi+\mathcal P\phi_{n}- \mu_{n})-g(\bar\psi)
= g'(\bar\psi)(\mathcal P\phi_{n}-\mu_{n})+{\mathfrak r}_{1}(\mathcal P\phi_{n}-\mu_{n})(\mathcal P\phi_{n}-\mu_{n}),
\end{split}
\end{equation}
where ${\mathfrak r}_{1}\in C(\mathbb R)$ satisfying ${\mathfrak r}_{1}(s)\to0$ as $s\to0.$ Inserting  \eqref{mucon110} into \eqref{mucon109}  gives
\begin{equation}\label{mucon111}
\mu_{n}\left(\int_{D}g'(\bar\psi)dx+\int_{D}\mathfrak r_{1}(\mathcal P\phi_{n}-\mu_{n})dx\right)=\int_{D}g'(\bar\psi)\mathcal P\phi_{n}dx+\int_{D}{\mathfrak r}_{1} (\mathcal P\phi_{n}-\mu_{n}) \mathcal P\phi_{n} dx.
\end{equation}
Denote
\begin{equation}\label{mucon1121}
\beta_{n}=\int_{D}\mathfrak r_{1}(\mathcal P\phi_{n}-\mu_{n})dx, \quad\gamma_{n}=\int_{D}{\mathfrak r}_{1} (\mathcal P\phi_{n}-\mu_{n}) \mathcal P\phi_{n} dx.
\end{equation}
Then \eqref{mucon111} becomes
\begin{equation}\label{mucon191}
\mu_{n} \int_{D}g'(\bar\psi)dx+\mu_n\beta_n=\int_{D}g'(\bar\psi)\mathcal P\phi_{n}dx+\gamma_n.
\end{equation}
By \eqref{muto0} and \eqref{wnbdd},   it is easy to check that
\begin{equation}\label{mucon112}
\lim_{n\to+\infty}\beta_{n}=0.
\end{equation}
\begin{equation}\label{mucon113}
\lim_{n\to+\infty}\frac{\gamma_{n}}{\|\mathcal P\phi_{n}\|_{L^{2}(D)}}=0  \quad {\mbox{(note that $\mathcal P\phi_n\neq 0$ by \eqref{pofthm22}).}}
\end{equation}
By \eqref{mucon112} and our  assumption \eqref{gpied0}, for  sufficiently large $n$, it holds  that
\[\int_{D}g'(\bar\psi)dx+\beta_{n}>0.\]
Therefore, from \eqref{mucon191}, $\mu_n$ can be written as
\begin{equation}\label{mucon114}
\mu_{n}  =\frac{\int_{D}g'(\bar\psi)\mathcal P\phi_{n}dx+\gamma_{n}}{\int_{D}g'(\bar\psi)dx+\beta_{n}}.
\end{equation}
To prove \eqref{muco489}, we compute
\begin{equation}\label{mucon1190}
\begin{split}
\left|\mu_{n}-\frac{\int_{D}g'(\bar\psi)\mathcal P\phi_{n}dx}{\int_{D}g'(\bar\psi)dx} \right|&=\left|\frac{\int_{D}g'(\bar\psi)\mathcal P\phi_{n}dx+\gamma_{n}}{\int_{D}g'(\bar\psi)dx+\beta_{n}}-\frac{\int_{D}g'(\bar\psi)\mathcal P\phi_{n}dx}{\int_{D}g'(\bar\psi)dx} \right|\\
&=\left|\frac{ \gamma_{n} \int_{D}g'(\bar\psi)dx-\beta_{n}\int_{D}g'(\bar\psi)\mathcal P\phi_{n}dx }{\int_{D}g'(\bar\psi)dx\left(\int_{D}g'(\bar\psi)dx+\beta_{n}\right)}\right|\\
&\le C\left(\beta_{n}\|\mathcal P\phi_{n}\|_{L^{2}(D)}+\gamma_{n}\right),
\end{split}
\end{equation}
where $C>0$ does not depend on $n$.      Combining \eqref{mucon112}, \eqref{mucon113} and \eqref{mucon1190}, we get \eqref{muco489}.

  Note that \eqref{muco489} implies
\begin{equation}\label{pofthm3io}
|\mu_{n}|\leq C\|\mathcal P\phi_{n}\|_{L^2(D)}
\end{equation}
for some $C>0$ depending not on $n$. This can be easily verified by the H\"older's inequality.

Inserting \eqref{muco489} into \eqref{pofthm2101}, after some algebraic operations we have that
\begin{equation}\label{pofthm2106}
\begin{split}
 \int_{D}\phi_{n}\mathcal P\phi_{n}dx- \int_{D}g'(\bar\psi)(\mathcal P\phi_{n})^{2}dx+\frac{\left(\int_{D}g'(\bar\psi)\mathcal P\phi_{n}\right)^{2}}{\int_{D}g'(\bar\psi)dx}
 \leq  \alpha_{n}^{2}\int_{D}g'(\bar\psi)dx +\int_{D}\vartheta_{n}(\mathcal P\phi_{n}-\mu_{n})^{2}dx.
 \end{split}
\end{equation}
For the second integral on the right-hand side of \eqref{pofthm2106}, it is easy to check that
\begin{equation}\label{pofthm2io}
\left|\int_{D}\vartheta_{n}(\mathcal P\phi_{n}-\mu_{n})^{2}dx\right|\leq C\|\vartheta_{n}\|_{L^{\infty}(D)}(\|\mathcal P\phi_{n}\|^{2}_{L^{2}(D)}+\mu_{n}^{2}),
\end{equation}
where $C>0$ does not depend on $n$.
Combining \eqref{thto0}, \eqref{muco48u}, \eqref{pofthm3io} and \eqref{pofthm2io}, we deduce that
\[\lim_{n\to+\infty}\frac{\alpha_{n}^{2}\int_{D}g'(\bar\psi)dx +\int_{D}\vartheta_{n}(\mathcal P\phi_{n}-\mu_{n})^{2}dx}{\|\mathcal P\phi_{n}\|_{L^{2}(D)}^{2}}=0,\]
which  leads to an obvious contradiction to \eqref{lemc1220}.

\end{proof}

\begin{remark}
Regarding the condition \eqref{lemc1},
Wolansky and Ghil proved a general stability criterion  in \cite{WG}.  See also \cite{LZ}. However, Wolansky and Ghil's result requires  $\min_{\bar D}g'(\bar\psi)>0$, hence can not be applied directly in the present paper.

\end{remark}

\begin{remark}
To ensure that Theorem \ref{thm2} holds, the condition \eqref{min0w} is indispensable. In fact, if $\bar \omega\in L^p(D)$ is a local maximizer of $E(\cdot,\mathbf a)$ relative to $\mathcal R_{\bar \omega},$ then there exists some increasing function $g$ such that
$\bar \omega= g(\mathcal P\bar\omega+{h_\mathbf a})$ almost everywhere in $D$. 
This can be proved by combining Theorem 3.3(iii) in \cite{B2} and Lemma 2.6 in \cite{BM}.

\end{remark}

\section{Compactness}\label{compac}

The purpose of this section is to prove a compactness result related to the variational characterization in Section \ref{varcha}, which is also an essential step in the proof of nonlinear stability.

Throughout this section, let $1<p<+\infty$ be fixed. For $w\in L^{p}(D)$ and  $r>0$, denote by $\mathbf B_{r}(w)$ the ball in $L^p(D)$ centered at $w$ with radius $r$, i.e.,
 \[\mathbf B_{r}(w):=\{v\in L^{p}(D)\mid \|v-w\|_{L^{p}(D)}<r\}.\]

The main result in this section can be stated as follows.
\begin{theorem}\label{prop41}
Let $\bar\omega$ and $\mathbf a$ be as in Section \ref{varcha}.   Then there exists some $\hat r>0,$  such that for any sequence
  $\{w_{n}\}\subset  \bar{\mathcal R}_{\bar\omega}\cap \mathbf B_{\hat r}{(\bar\omega)}$ satisfying
\[ \lim_{n\to+\infty} E(w_n,\mathbf a)=  E(\bar\omega,\mathbf a),\]
it holds that
$w_{n}\to \bar\omega$ in  $L^{p}(D)$ as $n\to+\infty$. Here and henceforth, $\bar{\mathcal R}_{\bar\omega}$ denotes the weak closure of $\mathcal R_{\bar\omega}$ in $L^p(D).$

\end{theorem}

The proof is very similar to that of Theorem 5 in Burton \cite{B5}. It relies on the strict convexity of $E(\cdot,\mathbf a)$ and the properties of rearrangements.
For completeness, we prove Theorem \ref{prop41} in detail below.

We begin with several lemmas.
\begin{lemma}\label{lem4030}
The strong and weak topologies of $L^{p}(D)$ restricted on $\mathcal R_{\bar\omega}$ are the same.
\end{lemma}
\begin{proof}
We only need to show that any strongly closed set $F\subset\mathcal R_{\bar\omega}$ is also weakly closed. Since the weak topology on $\mathcal R_{\bar\omega}$ is metrizable (see  for example \cite{Bre}, p. 74), it is enough to show that $F$ is sequentially closed on $\mathcal R_{\bar\omega}$. Let $\{w_{n}\}\subset F$ be a sequence such that $w_{n}$ converges to some $w\in\mathcal R_{\bar\omega}$ weakly  in $L^{p}(D)$ as $n\to+\infty$. It is clear that $\|w\|_{L^{p}(D)}=\|w_{n}\|_{L^{p}(D)}$ for every $n$. By uniform convexity, we have that $w_{n}$ converges to $w$ strongly in $L^{p}(D)$ as $n\to+\infty$. Therefore $w\in F$ since $F$ is strongly closed.
\end{proof}

\begin{lemma}\label{lem403}
There exists some weakly open set $U$ of $L^{p}(D)$ such that $ \mathcal R_{\bar\omega}\cap U=\mathcal R_{\bar\omega}\cap \mathbf B_{r_{0}}(\bar\omega),$ where $r_0$ is as in Theorem \ref{thm2}.
\end{lemma}
\begin{proof}
It is an easy consequence of Lemma \ref{lem4030}.
\end{proof}

\begin{lemma}\label{lem404}
There exists some $0<r_{1}<r_{0},$ such that
\[E(\bar\omega, {\mathbf a})\geq E(w, {\mathbf a}),\quad\forall\,w\in \bar{\mathcal R}_{\bar\omega}\cap \mathbf B_{r_{1}}(\bar\omega). \]
\end{lemma}
\begin{proof}
Let $U$ be as in Lemma \ref{lem403}.
Clearly $U$ is a strongly open set of $L^{p}(D)$ containing $\bar\omega$, which implies
the existence of some $r_{1}>0$ such that $\mathbf B_{r_{1}}(\bar\omega)\subset U.$ Without loss of generality,  we assume that $r_{1}<r_{0}.$

For any $w\in  \bar{\mathcal R}_{\bar\omega}\cap \mathbf B_{r_{1}}(\bar\omega)$, take a sequence $\{w_{n}\}\subset \mathcal R_{\bar\omega}$ such that $w_{n}\rightharpoonup w$ in $L^{p}(D).$  Since $\mathbf B_{r_{1}}(\bar\omega)\subset U$ and $U$ is weakly open on $\mathcal R_{\bar\omega},$ we have that $w_{n}\in U$ if $n$ is sufficiently large. Hence $w_{n}\in \mathcal R_{\bar\omega}\cap \mathbf B_{r_{0}}(\bar\omega)$ if $n$ is sufficiently large. Then the conclusion of Theorem \ref{thm2} gives
\begin{equation}\label{40331}
E(w_{n}, {\mathbf a})\leq E(\bar\omega, {\mathbf a}),
\end{equation}
provided that $n$ is  large enough. By the weak continuity of $E(\cdot,\mathbf a)$ (see Lemma \ref{lipcoe}), we can pass  to the limit $n\to+\infty$ in \eqref{40331} to obtain
\[E (w,{\mathbf a})\leq E (\bar\omega,{\mathbf a}).\]
\end{proof}
\begin{lemma}\label{lem405}
Let $r_{1}$ be as in Lemma \ref{lem404}.  Then for any $w\in \bar{\mathcal R}_{\bar\omega}\cap \mathbf B_{r_{1}}(\bar\omega)$ satisfying
$ E(w, {\mathbf a})= E (\bar\omega,{\mathbf a}),$
it holds that
$w=\bar\omega.$
\end{lemma}
\begin{proof}
Fix $w\in \bar{\mathcal R}_{\bar\omega}\cap \mathbf B_{r_{1}}(\bar\omega)$ such that
$ E(w, {\mathbf a})= E (\bar\omega,{\mathbf a}).$
By Lemma \ref{lem201}, for any $v\in \bar{\mathcal R}_{\bar\omega}$ and $\theta\in [0,1]$, we have
\[\theta v+(1-\theta)w\in\bar{\mathcal R}_{\bar\omega}.\]
On the other hand, it is also clear that $\theta v+(1-\theta)w\in \mathbf B_{r_{1}}(\bar\omega)$ if $\theta$ is small enough.
Applying Lemma \ref{lem404} gives
\[E (\theta v+(1-\theta)w,{\mathbf a})\leq E (\bar\omega,{\mathbf a})=E (w,{\mathbf a}).\]
Hence
\begin{equation}\label{40332}
\frac{d}{d\theta}E(\theta v+(1-\theta)w, {\mathbf a})\bigg|_{\theta=0^{+}}\leq 0.
\end{equation}
By a straightforward computation, we get from \eqref{40332} that
\begin{equation}\label{40333}
\int_{D}(\mathcal Pw+h_{\mathbf a})vdx\leq \int_{D}(\mathcal Pw+h_{\mathbf a})wdx.
\end{equation}
Note that \eqref{40333} holds for any $v\in\bar{\mathcal R}_{\bar\omega}$.

By Lemma \ref{lem202}, there exists some $\tilde w\in\mathcal R_{\bar\omega}$ such that
\begin{equation}\label{40334}
\int_{D}(\mathcal Pw+h_{\mathbf a})vdx\leq \int_{D}(\mathcal Pw+h_{\mathbf a})\tilde wdx,\quad \forall\,v\in\bar{\mathcal R}_{\bar\omega}.
\end{equation}
In particular,
\begin{equation}\label{40334}
\int_{D}(\mathcal Pw+h_{\mathbf a})wdx\leq \int_{D}(\mathcal Pw+h_{\mathbf a})\tilde wdx.
\end{equation}

We claim that $w=\tilde w$ (and thus $w\in\mathcal R_{\bar\omega}$). Define $w_{\theta}=\theta\tilde w+(1-\theta)w,$ where $\theta$ is a fixed small positive number such that $w_{\theta}\in \bar{\mathcal R}_{\bar\omega}\cap \mathbf B_{r_{1}}(\bar\omega)$.   By Lemma \ref{lem404},\begin{equation}\label{40335}
E(w_{\theta}, {\mathbf a})\leq E(\bar\omega, {\mathbf a})=E(w, {\mathbf a}).
\end{equation}
On the other hand, we compute
\begin{equation}\label{40336}
\begin{split}
E(w_{\theta}, {\mathbf a})- E (w,{\mathbf a})&=\frac{1}{2}\int_{D}w_{\theta}\mathcal Pw_{\theta}-w\mathcal Pwdx+\int_{D}h_{\mathbf a}(w_{\theta}-w)dx\\
&=\frac{1}{2}\int_{D}(w_{\theta}-w)\mathcal P(w_{\theta}-w)dx+\int_{D}(w_{\theta}-w)(\mathcal Pw+h_{\mathbf a})dx\\
&\geq \int_{D}(w_{\theta}-w)(\mathcal Pw+h_{\mathbf a})dx\\
&=\theta \int_{D}(\mathcal Pw+h_{\mathbf a})\tilde wdx- \theta\int_{D}(\mathcal Pw+h_{\mathbf a}) wdx \\
&\geq 0.
\end{split}
\end{equation}
Here we used \eqref{40334} and the fact that $\mathcal P$ is symmetric and positive (see Lemma \ref{eiyoyo}(ii)(iii)).
Therefore the inequalities in \eqref{40335} and \eqref{40336} are all equalities, which gives
\[\int_{D}(w_{\theta}-w)\mathcal P(w_{\theta}-w)dx=0.\]
By the positive definiteness of $\mathcal P$, we have $w_{\theta}=w$, and thus $w=\tilde w.$

To conclude, $w$ satisfies
\begin{equation}
w\in  {\mathcal R}_{\bar\omega}\cap \mathbf B_{r_{1}}(\bar\omega)\subset {\mathcal R}_{\bar\omega}\cap \mathbf B_{r_{0}}(\bar\omega), \quad E (w,{\mathbf a})=E(\bar\omega, {\mathbf a}).
\end{equation}
Then by Theorem \ref{thm2} we have $w=\bar\omega$.
\end{proof}

Now we are ready to prove Theorem \ref{prop41}.

\begin{proof}[Proof of  Theorem \ref{prop41}]
Take $\hat r=r_{1}/2.$  Since $\{w_{n}\}\subset  \bar{\mathcal R}_{\bar\omega}$, it is easy to see that
\begin{equation}\label{cymo}
\|w_n\|_{L^p(D)}\leq \|\bar\omega\|_{L^p(D)},\quad\forall\,n.
\end{equation}
Up to a subsequence,   $\{w_{n}\}$ has a weak limit, say $\eta$, in $L^{p}(D).$
Below we show that
$\eta=\bar\omega.$

First, it is clear  that $\eta\in \bar{\mathcal R}_{\bar\omega}$. Moreover, by  weak lower semicontinuity of the $L^p$ norm,
\[\|\eta-\bar\omega\|_{L^{p}(D)}\leq \liminf_{j\to+\infty}\|w_{n}-\bar\omega\|_{L^{p}(D)}\leq \hat r<r_{1},\]
which implies
\begin{equation}\label{lim45}
\eta\in \bar{\mathcal R}_{\bar\omega}\cap\mathbf B_{r_{1}}(\bar\omega).
\end{equation}
On the other hand, by  the weak continuity of $E(\cdot,{\mathbf a})$ in $L^{p}(D)$, we have that
\begin{equation}\label{lim459}
E (\eta,{\mathbf a})=\lim_{n\to+\infty}E(w_n, {\mathbf a})=E (\bar\omega,{\mathbf a}).
\end{equation}
Applying Lemma \ref{lem405}, we get from
 \eqref{lim45} and \eqref{lim459} that $\eta=\bar\omega$.

 Hence we have proved that  $w_n\rightharpoonup\bar\omega$ in $L^p(D)$ up to a subsequence.
 By weak lower semicontinuity, we have that
 \begin{equation}
 \|\bar\omega\|_{L^p(D)}\leq\liminf_{n\to+\infty}\|w_n\|_{L^p(D)}.
 \end{equation}
 Taking into account \eqref{cymo}, we obtain
  \begin{equation}
 \lim_{n\to+\infty}\|w_n\|_{L^p(D)}= \|\bar\omega\|_{L^p(D)}.
 \end{equation}
By uniform convexity,   we finally have that $w_n\to\bar\omega$ in $L^p(D)$  up to a subsequence.
Arguing by contradiction, we can  further  prove that the convergence actually holds for the whole sequence $\{w_{n}\}$.
\end{proof}

\section{Proof of nonlinear stability (Theorem \ref{thm1})}\label{nonlin}

With Theorem \ref{prop41} at hand, we can give the proof of Theorem \ref{thm1} now.  The basic idea is to make proper use of conserved quantities of the 2D Euler equations.

Fix $1<p<+\infty$. Let $\bar\psi$ be as in Theorem \ref{thm1}, and $\bar\omega$ be the corresponding vorticity. Then $\bar\omega$ satisfies the compactness property stated in Theorem \ref{prop41}.

Below we prove Theorem \ref{thm1} by contradiction.
Suppose  the conclusion in  Theorem \ref{thm1}  is false. Then, by Remark \ref{vico1},  there exist  some $\varepsilon_{0}>0,$  a sequence of classical Euler flows with vorticity $\{\omega_{n}\}$  and circulation vector $\{\mathbf a_{n}\}$, and a sequence of moments $\{t_n\}$, such that
\begin{equation}\label{pfth3}
\lim_{n\to+\infty}\|\omega_{n}(0,\cdot) -\bar\omega\|_{L^{p}(D)}=0,
\end{equation}
\begin{equation}\label{pfth300}
\lim_{n\to+\infty}|\mathbf a_{n}- {\mathbf a}|=0,
\end{equation}
\begin{equation}\label{pfth5}
\|\omega_{n}(t_{n},\cdot)-\bar\omega\|_{L^{p}(D)}\geq \varepsilon_{0},\quad\forall\,n.
\end{equation}
By choosing a smaller $\varepsilon_0$, we can assume that
\begin{equation}\label{pfth603}
\varepsilon_{0}< {\hat r},
\end{equation}
where $\hat r$ is the positive number in Theorem \ref{prop41}.
In addition, since it is clear that $\omega_{n}\in C([0,+\infty); L^{p}(D))$,   we can choose a new sequence of moments, still denoted by $\{t_n\}$, such that
\begin{equation}\label{pfth6}
\|\omega_{n}(t_{n},\cdot)-\bar\omega\|_{L^{p}(D)}=\varepsilon_{0},\quad\forall\,n.
\end{equation}
From \eqref{pfth603} and \eqref{pfth6}, we have
\begin{equation}\label{blamf1}
 \omega_n(t_n,\cdot)\in \mathbf B_{\hat r}(\bar\omega),\quad\forall\,n.
\end{equation}

\begin{lemma}\label{naf1}
It holds that
\begin{equation}\label{elig1}
\lim_{n\to+\infty} E(\omega_{n}(t_{n},\cdot),   \mathbf a )=E(\bar\omega,\mathbf a).
\end{equation}
\end{lemma}
\begin{proof}
We compute  that
\begin{align*}
&|E(\omega_{n}(t_{n},\cdot),\mathbf a)-E(\bar\omega,\mathbf a)|\\
\leq& |E(\omega_{n}(t_{n},\cdot),\mathbf a)-E(\omega_{n}(t_{n},\cdot),\mathbf a_n)|+|E(\omega_{n}(t_{n},\cdot),\mathbf a_n)-E(\bar\omega,\mathbf a)|\\
=&|E(\omega_{n}(t_{n},\cdot),\mathbf a)-E(\omega_{n}(t_{n},\cdot),\mathbf a_n)|+|E(\omega_{n}(0,\cdot),\mathbf a_n)-E(\bar\omega,\mathbf a)|.
\end{align*}
Here we used the conservation of  kinetic energy
\begin{equation*}
 E(\omega_{n}(t_{n},\cdot),\mathbf a_n)=E(\omega_{n}(0,\cdot),\mathbf a_n),\quad\forall\,n.
\end{equation*}
By \eqref{pfth3}, \eqref{pfth300} and the obvious fact that $\{\omega_n(t_n,\cdot)\}$ is bounded in $L^p(D)$, we can apply  Lemma \ref{lipcoe} to get
  \[\lim_{n\to+\infty}|E(\omega_{n}(t_{n},\cdot),\mathbf a)-E(\omega_{n}(t_{n},\cdot),\mathbf a_n)|=0,\quad    \lim_{n\to+\infty}|E(\omega_{n}(0,\cdot),\mathbf a_n)-E(\bar\omega,\mathbf a)|=0.\]
 Hence \eqref{elig1} is proved.
\end{proof}

By now, we know that the sequence $\{\omega_{n}(t_{n},\cdot)\}$ satisfies \eqref{blamf1} and \eqref{elig1}. If additionally $\{\omega_{n}(t_{n},\cdot)\}\subset \bar{\mathcal R}_{\bar\omega}$, then by Theorem \ref{prop41} we have that $\omega_{n}(t_{n},\cdot)\to\bar\omega$ in $L^p(D)$ as $n\to+\infty$, a contradiction to \eqref{pfth6}. In other words, if the vorticities of the perturbed flows belong to $\bar{\mathcal R}_{\bar\omega}$, then stability has been proved.

To deal with perturbations off $\bar{\mathcal R}_{\bar\omega}$, we use the method of ``followers" introduced by Burton in \cite{B5}.


\begin{lemma}\label{follower}
For every $\omega_{n}$, there exists a ``follower'' $\sigma_{n}:[0,+\infty)\times D\to \mathbb R$,  satisfying
\begin{equation}\label{pfth8}
\sigma_{n}(t,\cdot)\in  {\mathcal R}_{\bar\omega},\quad \forall\,t>0,
\end{equation}
\begin{equation}\label{pfth81}
\sigma_{n}(t,\cdot)-\omega_{n}(t,\cdot)\in\mathcal R_{\bar\omega-\omega_{n}(0,\cdot)},\quad\forall\,t>0.
\end{equation}
\end{lemma}
\begin{proof}

Denote by $\mathbf v_n$ the velocity field of the Euler flow determined by $\omega_n$ and $\mathbf a_n$, or equivalently, $\mathbf v_n=\nabla^\perp(\mathcal P\omega_n+h_{\mathbf a_n}).$
Let $\Phi_{n}:[0,+\infty)\times D\to D$ be the flow map related to $\mathbf v_n$,  i.e.,
\begin{equation}\label{flowmap}
\frac{d\Phi_{n}(t,x)}{dt}=\mathbf v_n(t,\Phi_{n}(t,x)), \,\,t\in\mathbb R;\quad
\Phi_{n}(0,x)=x,\,\,\forall\,x\in  D.
\end{equation}
Note that \eqref{flowmap} has a global solution $\Phi_{n}\in C^{1}([0,+\infty)\times D)$ since $\mathbf v_n\in C^{1}([0,+\infty)\times   D)$. Moreover, for  fixed $t\geq 0,$ $\Phi_{n}(t,\cdot)$ is one-to-one from $D$ onto $D$.
Denote by $\Psi_{n}(t,\cdot)$ the inverse map of $\Phi_{n}(t,\cdot)$. Define
\begin{equation}\label{naf3}
\sigma_{n}(t,x)=\bar\omega(\Psi_{n}(t,x)).
\end{equation}
Below we show that $\sigma_{n}$ satisfies \eqref{pfth8} and \eqref{pfth81}.

Since $\mathbf v_{n}$ is obviously divergence-free, by the Liouville theorem (see \cite{MPu}, p. 48), $\Phi_{n}$, and thus $\Psi_{n}$, are area-preserving, which implies that
\begin{equation}\label{pfth8i}
\sigma_{n}(t,\cdot)\in  {\mathcal R}_{\bar\omega},\quad \forall\,t>0.
\end{equation}
On the other hand, by the vorticity-transport formula (see \cite{MB}, p.20),
\begin{equation}\label{naf4}
\omega_{n}(t,x)=\omega_{n,0}(\Psi_{n}(t,x)), \quad \omega_{n,0}:=\omega_{n}(0,\cdot).
\end{equation}
From \eqref{naf3} and \eqref{naf4},
we obtain
\[(\sigma_{n}-\omega_{n})(t,x)=(\bar\omega-\omega_{n,0})(\Psi_{n}(t,x)).\]
Therefore
\[\sigma_{n}(t,\cdot)-\omega_{n}(t,\cdot)\in\mathcal R_{\bar\omega-\omega_{n,0}},\quad\forall\,t>0.\]
\end{proof}

\begin{lemma}\label{naf71}
It holds that
\begin{equation}\label{pfth90}
\lim_{n\to+\infty}\|\sigma_{n}(t_{n},\cdot)-\omega_{n}(t_{n},\cdot)\|_{L^{p}(D)}=0.
\end{equation}
\end{lemma}

\begin{proof}
By \eqref{pfth81}, we have that
\begin{equation}\label{pfth81p}
\|\sigma_{n}(t,\cdot)-\omega_{n}(t,\cdot)\|_{L^p(D)}=\|\bar\omega-\omega_n(0,\cdot)\|_{L^p(D)},\quad \forall\,t>0.
\end{equation}
Hence the desired result follows from \eqref{pfth3} and \eqref{pfth81p}.
\end{proof}

\begin{lemma}
There exists some positive integer $N_0$, such that for any $n\geq N_0$, it holds that
\begin{equation}\label{pfth16}
 \sigma_{n}(t_{n},\cdot)\in \mathbf B_{\hat r}(\bar\omega).
\end{equation}
\end{lemma}
\begin{proof}
By \eqref{pfth603}, \eqref{pfth6} and \eqref{pfth90}, we have that
\begin{equation}\label{pfth16t}
\|\sigma_{n}(t_{n},\cdot) - \bar\omega\|_{L^{p}(D)}\leq \|\sigma_{n}(t_{n},\cdot) - \omega_{n}(t_{n},\cdot)\|_{L^{p}(D)}+\| \omega_{n}(t_{n},\cdot)- \bar\omega\|_{L^{p}(D)}<\hat r,
\end{equation}
provided that $n$ is large enough.
\end{proof}

\begin{lemma}
It holds that
\begin{equation}\label{pfth15}
\lim_{n\to+\infty}E(\sigma_{n}(t_{n},\cdot),\mathbf a)=E(\bar\omega,\mathbf a).
\end{equation}
\end{lemma}
\begin{proof}
First, by  Lemma  \ref{lipcoe} and \eqref{pfth90}, we  have that
\begin{equation}\label{pfth13}
 \lim_{n\to+\infty}|E(\sigma_{n}(t_{n},\cdot),\mathbf a)- E(\omega_{n}(t_{n},\cdot),\mathbf a)|= 0.
 \end{equation}
 Then \eqref{pfth15} follows from \eqref{elig1} and \eqref{pfth13}.

\end{proof}

Now by \eqref{pfth8}, \eqref{pfth16} and \eqref{pfth15}, we can apply Theorem \ref{prop41} to obtain  that $\sigma_{n}(t_{n},\cdot)\to\bar\omega$ in $L^p(D)$ as $n\to+\infty,$ which in combination with \eqref{pfth90} gives
\begin{equation}\label{pfth9099}
\lim_{n\to+\infty}\|\omega_{n}(t_{n},\cdot)-\bar\omega\|_{L^{p}(D)}=0.
\end{equation}
This obviously contradicts \eqref{pfth6}. Therefore Theorem \ref{thm1} is proved.

 \begin{remark}\label{lessregular}
We briefly discuss to what extent our stability result holds for less regular perturbations.
 For an Euler flow, the vorticity $\omega$ and the circulation vector $\mathbf b$ formally satisfy the following nonlinear transport equation
\begin{equation}\label{vore}
\partial_t\omega+\nabla\cdot\left(\omega\nabla^\perp(\mathcal P\omega+h_\mathbf b)\right)=0,\quad t>0,\,\,x\in D.
 \end{equation}
If  $\omega\in L^\infty_{\rm loc}((0,+\infty); L^p(D))$ with $4/3\leq p<+\infty$,  we can interpret \eqref{vore} in the  following distributional sense:
\[\int_0^{+\infty}\int_D\partial_t\zeta \omega+\omega\nabla^\perp(\mathcal P\omega+h_{\mathbf b})\cdot\nabla\zeta dxdt,\quad \forall\,\zeta\in C_c^\infty((0,+\infty)\times D).\]
Let $\bar\omega\in L^p(D)$ and $\mathbf a\in\mathbb R^N$. Suppose $\bar\omega$ is an isolated local maximizer of $E(\cdot,\mathbf a)$ relative $\mathcal R_{\bar\omega}$. Repeating  the arguments in Sections 5 and 6, we can prove the following  assertion:

For any $\varepsilon>0$, there exists some $\delta>0$, such that for any distributional solution $\omega$ of \eqref{vore} satisfying
 \begin{itemize}
 \item[(a)] $\omega\in C([0,+\infty);L^p(D))$,
 \item[(b)] $\omega(t,\cdot)\in\mathcal R_{\omega(0,\cdot)}$ for all $t>0$,
 \item[(c)]$E(\omega(0,\cdot),\mathbf b)=E(\omega(t,\cdot),\mathbf b)$  for all $t>0$,
 \item[(d)] there exists a follower of $\omega$ as in Lemma \ref{follower},
 \end{itemize}
it holds that
\[ \|\omega(0,\cdot)-\bar{\omega}\|_{L^{p}(D)}+|\mathbf b-\mathbf a|<\delta\quad\Longrightarrow\quad \|\omega(t,\cdot)-\bar{\omega}\|_{L^{p}(D)}<\varepsilon \quad \forall\,t>0.\]
The existence of a distributional solution of \eqref{vore} satisfying (a)-(d) is an interesting problem worth studying, but is beyond the purpose of this paper. When $p\geq 2$ and $D$ is simply-connected (corresponding to $\mathcal P=\mathcal G$ and $h_{\mathbf b}=0$), a complete answer was given in \cite{B5}.
  \end{remark}

\appendix

\section{Proof of Theorem \ref{thm0} }\label{apdx1}

We first give some assumptions and notation. Let $D,$ $\bar\psi$ and $g$ be as in Theorem \ref{thm0}.
Denote
\[\bar m=\min_{\bar D}\bar\psi,\quad \bar M=\max_{\bar D}\bar\psi.\]
As in Section \ref{subsec51},  we can redefine $g$ outside  $[\bar m,\bar M]$ as follows:
\[g(s)=
\begin{cases}
g'(\bar m)(s-\bar m)+g(\bar m)&s<\bar m,\\
g'(\bar M)(s-\bar M)+g(\bar M) &s>\bar M.
\end{cases}
\]
Then $g\in C^{1}(\mathbb R)$,  and satisfies
\begin{equation}\label{cmm3}
 \min_{\mathbb R}g'=\min_{\bar D}g'(\bar\psi),\quad \max_{\mathbb R}g'=\max_{\bar D}g'(\bar\psi).
 \end{equation}
 In addition, the inverse of $g$ exists.
  Denote $f=g^{-1}.$ It is clear that $f\in C^{1}(\mathbb R)$, and
\begin{equation}\label{appb1}
\min_{\mathbb R}f'=\frac{1}{\max_{\bar D}g'(\bar\psi)},\quad \max_{\mathbb R}f'=\frac{1}{\min_{\bar D}g'(\bar\psi)}.
\end{equation}

Define
\[\mathcal H(u)=\frac{1}{2}\int_{D}|\nabla u|^{2}dx-\int_{D}F(-\Delta u)dx,\quad \mbox{where } F(s)=\int_{0}^{s}f(\tau)d\tau.\]
By Lemma \ref{lt}(ii),   $\mathcal  H$ is in fact the stream function version of  $EC$  (up to some suitable constant).

\begin{proof}[Proof of Theorem \ref{thm0}]
We only need to prove \eqref{conc3}.
Let $\bar\psi$ and $\varphi$ be as in Theorem \ref{thm0}.   For simplicity, denote $\varphi_{t}=\varphi(t,\cdot)$.
 By energy and vorticity conservation, we have
 \[\mathcal  H(\bar\psi+\varphi_t)=\mathcal H(\bar\psi+\varphi_{0}),\quad\forall\,t>0,\]
 which implies that
\begin{equation}\label{hzy1}
\mathcal H(\bar\psi+\varphi_t)-\mathcal H(\bar\psi)=\mathcal H(\bar\psi+\varphi_{0})-H(\bar\psi),\quad\forall\,t>0.
\end{equation}
Denote $\bar\omega=-\Delta\bar\psi,$ $\phi_{t} =-\Delta\varphi_{t}.$ After a simple computation, the left-hand side of \eqref{hzy1} can be written as
\begin{equation}\label{appb2}
\mathcal H(\bar\psi+\varphi_t)-\mathcal H(\bar\psi)= \int_{D}\nabla\bar\psi\cdot\nabla\varphi_{t}dx+\frac{1}{2}\int_{D}|\nabla\varphi_{t}|^{2}dx-\int_{D}F(\bar\omega+\phi_{t})-F(\bar\omega)dx.
\end{equation}
Integrating by parts, we have that
\begin{equation}\label{appb3}
\begin{split}
 \int_{D}\nabla\bar\psi\cdot\nabla\varphi_{t}dx &=\int_{\partial D} \bar\psi\nabla\varphi_{t}\cdot d\mathbf n+\int_{D}\bar\psi\phi_{t} dx\\
 &=\sum_{i=1}^{N}\bar\psi|_{\Gamma_{i}}\int_{\Gamma_{i}} \nabla\varphi_{t}\cdot d\mathbf n+\int_{D}\bar\psi\phi_{t} dx\\
 &=\int_{D}\bar\psi\phi_{t} dx.
 \end{split}
\end{equation}
Note that we used the assumption that $\varphi\in \mathcal Y$ in the last equality of \eqref{appb3}.
On the other hand, by Taylor expansion,
\begin{equation}\label{appb4}
f(\bar\omega)\phi_{t}+\frac{1}{2}\min_{\mathbb R}f'\int_{D}\phi^{2}\leq F(\bar\omega+\phi_{t})-F(\bar\omega) \leq \int_{D}f(\bar\omega)\phi_{t}+\frac{1}{2}\max_{\mathbb R}f' \phi_{t}^{2}.
\end{equation}
Inserting \eqref{appb3} and \eqref{appb4} into \eqref{appb2}, and using the fact that $\bar\psi=f(\bar\omega)$, we have that
\begin{equation}\label{appb5}
\begin{split}
\frac{1}{2}\int_{D}|\nabla\varphi_{t}|^{2}dx-\frac{1}{2}\max_{\mathbb R}f' \int_{D}\phi_{t}^{2}dx
&\leq \mathcal H(\bar\psi+\varphi_t)-\mathcal H(\bar\psi)\\
&\leq \frac{1}{2}\int_{D}|\nabla\varphi_{t}|^{2}dx-\frac{1}{2}\min_{\mathbb R}f'\int_{D} \phi_{t}^{2}dx.
\end{split}
\end{equation}
Similarly, the right-hand side of \eqref{hzy1} satisfies the following estimate:
\begin{equation}\label{appb6}
\begin{split}
\frac{1}{2}\int_{D}|\nabla\varphi_{0}|^{2}dx-\frac{1}{2}\max_{\mathbb R}f' \int_{D}\phi_{0}^{2}dx
&\leq \mathcal H(\bar\psi+\varphi_0)-\mathcal H(\bar\psi)\\
&\leq \frac{1}{2}\int_{D}|\nabla\varphi_{0}|^{2}dx-\frac{1}{2}\min_{\mathbb R}f' \int_{D}\phi_{0}^{2}dx.
\end{split}
\end{equation}
Combining \eqref{hzy1}, \eqref{appb5} and \eqref{appb6}, we obtain
\begin{equation*}\label{appb8}
 \min_{\mathbb R}f'\int_{D} \phi_{t}^{2}dx-\int_{D}|\nabla\varphi_{t}|^{2}dx\leq \max_{\mathbb R}f'\int_{D} \phi_{0}^{2}dx-\int_{D}|\nabla\varphi_{0}|^{2}dx,\quad\forall\,t>0,
 \end{equation*}
or equivalently (recalling \eqref{appb1}),
\begin{equation}\label{appb9}
\frac{1}{\max_{\bar D}g'(\bar\psi)}\int_{D} \phi_{t}^{2}dx-\int_{D}|\nabla\varphi_{t}|^{2}dx\leq  \frac{1}{\min_{\bar D}g'(\bar\psi)}\int_{D} \phi_{0}^{2}dx-\int_{D}|\nabla\varphi_{0}|^{2}dx,\quad\forall\,t>0.
\end{equation}
Since $\varphi_{t}\in\mathcal Y$ for all $t> 0$,  we deduce from \eqref{mathsfc}  that
\begin{equation}\label{appb91}
\int_{D}|\nabla \varphi_{t}|^{2}dx\leq \frac{1}{\mathsf c} \int_{D}|\Delta\varphi_{t}|^{2}dx=\frac{1}{\mathsf c} \int_{D}\phi_{t}^{2}dx,\quad\forall\,t>0.
\end{equation}
Inserting \eqref{appb91} into \eqref{appb9}, we obtain
  \begin{equation*}\label{conc30}
 \left(\frac{1}{\max_{\bar D}g'(\bar\psi)}-\frac{1}{\mathsf c}\right)\int_{D}\phi_{t}^{2}dx \leq \frac{1}{\min_{\bar D}g'(\bar\psi)}\int_{D} \phi_{0}^{2}dx-\int_{D}|\nabla\varphi_{0}|^{2}dx,\quad\forall\,t>0,
\end{equation*}
which implies \eqref{conc3} immediately.
\end{proof}

\section{Properties of $\mathcal X, \mathcal Y$ and $\mathcal P$}\label{apdx2}

In this section, we deduce some properties concerning the two function spaces $\mathcal X, \mathcal Y$ and   the operator $\mathcal P$. Recall their definitions:
\[
\mathcal X=\left\{u\in H^{1}(D)\mid u=v+\sum_{i=1}^{N}\theta_{i}\zeta_{i},\, v\in H^{1}_0(D), \,\theta_{i}\in\mathbb R,\, i=1\cdot\cdot\cdot,N\right\},
\]
\[
\mathcal Y=\left\{u\in H^{2}(D)\mid u\in\mathcal X, \,\int_{\Gamma_{i}}\nabla u\cdot d\mathbf n=0, \,i=1\cdot\cdot\cdot,N\right\},
\]
\[\mathcal P\phi=\mathcal G\phi+\sum_{i,j=1}^{N}q_{ij}\int_{D}\zeta_{i}\phi dx\zeta_{j}.\]

\begin{lemma}\label{eiyoyo}
 For the operator $\mathcal P$, the following assertions hold:
 \begin{itemize}
 \item[(i)]For any $1<p<+\infty,$ $\mathcal P$ is a well-defined, bounded, linear operator from   $L^{p}(D)$ to $W^{2,p}(D).$
  \item[(ii)] $\mathcal P$ is symmetric, i.e.,
  \begin{equation}\label{psym}
  \int_{D}\phi_{1}\mathcal P\phi_{2}dx=\int_{D} \phi_{2}\mathcal P\phi_{1}dx,\quad \forall\,\phi_{1}, \phi_{2}\in L^{p}(D).
  \end{equation}
  Note that the integrals in \eqref{psym} make sense since $\mathcal P\phi\in W^{2,p}(D)\hookrightarrow C(\bar D)$ for any $\phi\in L^p(D)$.
    \item[(iii)] $\mathcal P$ is positive definite, i.e.,
  \begin{equation}\label{ppdef}
  \int_{D}\phi \mathcal P\phi dx\geq 0,\quad \forall\,\phi\in L^{p}(D),
  \end{equation}
  and the equality holds if and only if $\phi=0.$
  \end{itemize}
\end{lemma}

  \begin{proof}
Item (i) follows from the fact that $\mathcal G$ is a bounded linear operator from   $L^{p}(D)$ to $W^{2,p}(D)$ (by standard elliptic estimates). Items (ii) and (iii) follow from the facts that $\mathcal G$ is a symmetric and positive definite operator,  and $(q_{ij})$ a symmetric and positive definite  matrix.
\end{proof}

Recall the notation (N1)-(N5) introduced in Section 1.
  \begin{lemma}\label{propb1}
 For the space $\mathcal X$, the following assertions hold:
  \begin{itemize}
\item [(i)] For any $u\in \mathcal X,$ write $u=v+\sum_{i=1}^{N}\theta_{i}\zeta_{i}$, where  $v\in H^{1}_0(D), $ $\theta_{1},\cdot\cdot\cdot,\theta_N\in\mathbb R$. Then
\begin{equation}\label{xpp1}
\theta_{i}=\sum_{j=1}^{N}\left(q_{ij}\int_{D}\nabla u\cdot\nabla\zeta_{j}dx\right),\quad i=1,\cdot\cdot\cdot,N.
\end{equation}
 \item[(ii)] $\mathcal X$ is  closed, thus weakly closed, in $H^{1}(D).$
  \end{itemize}
 \end{lemma}
 \begin{proof}
 We first observe that
 \begin{equation}\label{xpp2022}
\int_{D}\nabla v\cdot\nabla \zeta_{i}dx=0,\quad \forall\,v\in H^1_0(D),\,\,i=1,\cdot\cdot\cdot,N.
\end{equation}
 In fact, choose a sequence $\{v_{n}\}\subset C_{c}^{\infty}(D)$ such that $v_{n}\to v$ in $H^{1}_{0}(D)$ as $n\to+\infty.$ Then by integration by parts,
\begin{equation}\label{xpp2}
\int_{D}\nabla v\cdot\nabla \zeta_{i}dx=\lim_{n\to+\infty}\int_{D}\nabla v_{n}\cdot\nabla \zeta_{i}dx=-\lim_{n\to+\infty}\int_{D} v_{n}\Delta \zeta_{i}dx=0.
\end{equation}

Now we prove (i).
Using \eqref{xpp2022}, we have that
\begin{equation}\label{xpp3}
\begin{split}
\int_{D}\nabla u\cdot\nabla\zeta_{i}dx&=\int_{D}\nabla v\cdot\nabla\zeta_{i}dx+\sum_{j=1}^{N}\left(\theta_{j}\int_{D}\nabla\zeta_{i}\cdot\nabla\zeta_{j} dx\right)\\
&=\sum_{j=1}^{N}\left(\theta_{j}\int_{D}\nabla\zeta_{i}\cdot\nabla\zeta_{j} dx\right)\\
&=\sum_{j=1}^{N}p_{ij}\theta_{j}.
\end{split}
\end{equation}
 Taking into account the fact that $(q_{ij})=(p_{ij})^{-1}$, we get \eqref{xpp1}.

To prove (ii), it suffices to show that for any sequence $\{u_{n}\}\subset \mathcal X,$ if $u_{n}\to u$  in $H^{1}(D)$ for some $u\in H^1(D)$, then   $u\in \mathcal X.$
To this end, we write
 \[u_{n}=v_{n}+\sum_{i=1}^{N}\theta_{n,i}\zeta_{i},\quad v_{n}\in H^{1}_{0}(D),\,\,\theta_{n,1},\cdot\cdot\cdot,\theta_{n,N}\in \mathbb R.\]
 By \eqref{xpp1}, we have that as $n\to+\infty$
 \begin{equation*}\label{xpp4}
 \theta_{n,i}=\sum_{j=1}^{N}\left(q_{ij}\int_{D}\nabla u_{n}\cdot\nabla\zeta_{j}dx\right)\to \sum_{j=1}^{N}\left(q_{ij}\int_{D}\nabla u\cdot\nabla\zeta_{j}dx\right) \quad\mbox{in $H^1(D)$}.
 \end{equation*}
Hence
\[\sum_{i=1}^{N}\theta_{n,i}\zeta_{i}\to\sum_{j=1}^{N}\left(q_{ij}\int_{D}\nabla u\cdot\nabla\zeta_{j}dx\right)\zeta_{i} \quad\mbox{in $C^{1}(\bar D)$ as $n\to+\infty$},\]
which implies that
 \[v_{n}=u_{n}-\sum_{i=1}^{N}\theta_{n,i}\zeta_{i}\to u-\sum_{j=1}^{N}\left(q_{ij}\int_{D}\nabla u\cdot\nabla\zeta_{j}dx\right)\zeta_{i}\quad\mbox{ in $H^{1}(D)$ as $n\to+\infty$}. \]
 Since $H^{1}_{0}(D)$ is  closed in $H^{1}(D)$, we infer that
\begin{equation}\label{7892}
u-\sum_{j=1}^{N}\left(q_{ij}\int_{D}\nabla u\cdot\nabla\zeta_{j}dx\right)\zeta_{i} \in H^{1}_{0}(D).
\end{equation}
Hence 
\[u=\left[u-\sum_{j=1}^{N}\left(q_{ij}\int_{D}\nabla u\cdot\nabla\zeta_{j}dx\right)\zeta_{i}\right] +\sum_{j=1}^{N}\left(q_{ij}\int_{D}\nabla u\cdot\nabla\zeta_{j}dx\right)\zeta_{i}\in\mathcal X.\]
Finally, since $\mathcal X$ is convex, we see that $\mathcal X$ is also weakly closed in $H^{1}(D)$.
  \end{proof}

   \begin{lemma}\label{propb2}
For the space $\mathcal Y$, the following assertions hold:
  \begin{itemize}
  \item[(i)] For any $\phi\in L^{2}(D)$ and $i=1,\cdot\cdot\cdot,N,$
 \begin{equation}\label{yp00}
 \int_{D}\nabla\zeta_{i}\cdot\nabla\mathcal P\phi dx=\int_{D}\zeta_{i}\phi dx=-\int_{\Gamma_{i}}\nabla\mathcal G\phi\cdot d\mathbf n.
 \end{equation}
 \item[(ii)]$\mathcal P$ is one-to-one from $L^{2}(D)$ onto $\mathcal Y$, and
 \begin{equation}\label{pcondel}
 \mathcal P^{-1}=-\Delta.
 \end{equation}
\item[(iii)]
For any $\phi\in L^{2}(D),$
\begin{equation}\label{yp301}
\int_{D}\phi\mathcal P\phi dx=\int_{D}|\nabla\mathcal P\phi|^{2}dx.
\end{equation}
As a consequence, for any $u\in\mathcal Y,$
\begin{equation}\label{ypx99}
\int_{D}-u\Delta u dx=\int_{D}|\nabla u|^{2}dx
\end{equation}

\end{itemize}
 \end{lemma}

  \begin{proof}
First we prove (i). By a density argument (recall item (i) of Lemma \ref{eiyoyo}), it is enough to prove \eqref{yp00} for any $\phi\in C_{c}^{\infty}(D).$ Fix $1\leq i\leq N.$ By a direct computation, we have that
 \begin{align*}
 \int_{D}\nabla\zeta_{i}\cdot\nabla\mathcal P\phi dx&=\int_{D}\nabla\zeta_{i}\cdot\nabla\left(\mathcal G\phi+\sum_{j,k=1}^{N} q_{jk}\int_{D}\zeta_{j}\phi dx\zeta_{k} \right)dx\\
 &=\sum_{j,k=1}^{N} \left(q_{jk}\int_{D}\nabla\zeta_{i}\cdot\nabla  \zeta_{k} dx\int_{D}\zeta_{j}\phi dx\right)\\
 &=\sum_{j,k=1}^{N} \left(q_{jk}p_{ik}\int_{D}\zeta_{j}\phi dx\right)\\
 &=\int_{D}\zeta_{i}\phi dx.
 \end{align*}
Here we used the fact that $\int_{D}\nabla\zeta_{i}\cdot\nabla\mathcal G\phi dx=0$ (see \eqref{xpp2022}). Hence the first equality in \eqref{yp00} has been proved.
For the second equality, by the Stokes theorem,
 \begin{align*}
\int_{D}\zeta_{i}\phi dx&=\int_{D}\zeta(-\Delta\mathcal G\phi)dx\\
&=-\int_{D}\nabla\cdot(\zeta_{i}\nabla\mathcal G\phi)dx+\int_{D}\nabla\zeta_{i}\cdot\nabla\mathcal G\phi dx\\
&=-\int_{\Gamma_{i}}\nabla\mathcal G\phi\cdot d\mathbf n.
 \end{align*}
Note that in the last equality we used  $\int_{D}\nabla\zeta_{i}\cdot\nabla\mathcal G\phi dx=0$ again.

Next we prove (ii).
We  first show that   $\mathcal P\phi\in\mathcal Y$ for any $\phi\in L^{2}(D).$ It is clear  that $\mathcal P\phi\in \mathcal X\cap H^{2}(D).$ Below we show that
 \begin{equation}\label{yp0}
 \int_{\Gamma_{i}}\nabla\mathcal P\phi\cdot d\mathbf n=0,\quad i=1,\cdot\cdot\cdot,N.
 \end{equation}
 Fix $1\leq i\leq N$. Again, by a density argument, it suffices to verify \eqref{yp0} for any $\phi\in C_{c}^{\infty}(D)$. By the Stokes theorem, we compute as follows:
  \begin{equation*}
 \int_{\Gamma_{i}}\nabla\mathcal P\phi\cdot d\mathbf n=\int_{\partial D}\zeta_{i}\nabla\mathcal P\phi\cdot d\mathbf n=\int_{D}\nabla\cdot(\zeta_{i}\nabla\mathcal P\phi)dx=\int_{D}\nabla\zeta_{i}\cdot\nabla\mathcal P\phi dx-\int_{D}\zeta_{i}\phi dx=0.
  \end{equation*}
 Here we used the obvious fact that $-\Delta(\mathcal P\phi)=\phi$ and the first equality in \eqref{yp00}.
To proceed, notice  that
 $-\Delta(\mathcal P\phi)=\phi$ for all $\phi\in L^{2}(D)$. Therefore, to prove \eqref{pcondel}, it is sufficient to show that
  \begin{equation}\label{yp200}
  \mathcal P(-\Delta u)=u,\quad\forall\,u\in\mathcal Y.
 \end{equation}
 To this end, for any $u\in\mathcal Y$ we write
  \begin{equation}\label{yp201}
  u=v+\sum_{i=1}^{N}\theta_{i}\zeta_{i},\quad v\in H^{1}_{0}(D),\quad \theta_{i}=\sum_{j=1}^{N}\left(q_{ij}\int_{D}\nabla u\cdot\nabla\zeta_{j}dx\right).
  \end{equation}
Here we used   \eqref{xpp1}. Then
 \begin{equation}\label{yp202}
 \begin{split}
 \mathcal P(-\Delta u)&=\mathcal G(-\Delta v)+\sum_{i,j=1}^{N}\left(q_{ij}\int_{D}\zeta_{j}(-\Delta u)dx\zeta_{i}\right)\\
 &=v+\sum_{i,j=1}^{N}\left(q_{ij}\int_{D}\zeta_{j}(-\Delta u)dx\zeta_{i}\right).
 \end{split}
 \end{equation}
Fix $1\leq j\leq N$. By the Stokes theorem, we have that
 \begin{equation}\label{yp203}
 \begin{split}
 \int_{D}\zeta_{j}(-\Delta u)dx&=-\int_{D}\nabla\cdot(\zeta_{j}\nabla u)dx+\int_{D}\nabla u\cdot\nabla\zeta_{j}dx\\
 &=-\int_{\partial D} \zeta_{j}\nabla u\cdot d\mathbf n+\int_{D}\nabla u\cdot\nabla\zeta_{j}dx\\
 &=-\sum_{k=0}^{N}\int_{\Gamma_{k}} \zeta_{j}\nabla u\cdot d\mathbf n+\int_{D}\nabla u\cdot\nabla\zeta_{j}dx\\
 &=\int_{D}\nabla u\cdot\nabla\zeta_{j}dx.
 \end{split}
 \end{equation}
 Note that we have used the following facts in the last equality  \[\zeta_{j}=0\quad\mbox{on }\Gamma_{0},\quad \int_{\Gamma_{k}}\nabla u\cdot d\mathbf n=0\quad\mbox{for any $1\leq k\leq N$ (since $u\in\mathcal Y$)}.
 \] Combining \eqref{yp201}, \eqref{yp202} and \eqref{yp203}, the claim \eqref{yp200} is proved.

 Finally we prove (iii). Notice that \eqref{yp301} and \eqref{ypx99} are in fact equivalent by item (ii). Below
we prove \eqref{yp301} for any $\phi\in C_c^\infty(D)$. By the Stokes theorem, we compute
 \begin{align*}
 \int_{D}|\nabla\mathcal P\phi|^{2}dx&=\int_{D}\nabla\cdot(\mathcal P\phi\nabla \mathcal P\phi)dx-\int_{D}\mathcal P\phi \Delta(\mathcal P\phi)dx\\
 &=\int_{\partial D}\mathcal P\phi\nabla \mathcal P\phi\cdot d\mathbf n+\int_{D}\phi\mathcal P\phi dx\\
 &=\sum_{i,j=1}^{N}\left(q_{ij}\int_{D}\phi \zeta_{i}dx\int_{\partial D}\zeta_{j}\nabla\mathcal P\phi \cdot d\mathbf n\right)+\int_{D}\phi\mathcal P\phi dx\\
 &=\sum_{i,j=1}^{N}\left(q_{ij}\int_{D}\phi \zeta_{i}dx\int_{\Gamma_{j}}\nabla\mathcal P\phi \cdot d\mathbf n\right)+\int_{D}\phi\mathcal P\phi dx\\
 &=\int_{D}\phi\mathcal P\phi dx.
 \end{align*}
Note that in the last equality we used the fact that $\int_{\Gamma_{j}}\nabla\mathcal P\phi \cdot d\mathbf n=0$ for any $1\leq j\leq N$ (since $\mathcal P\phi\in\mathcal Y$ by item (ii)). Hence \eqref{yp301} is proved.

 \end{proof}

 \section{Existence and uniqueness of an elliptic problem}\label{apdx3}

Recall the notation (N1)-(N5) in Section 1.
   \begin{proposition}\label{propc2}
    Let $D$ be a multiply-connected bounded domain of the form \eqref{dform}. Let $\omega\in L^2(D)$ and $\mathbf a\in \mathbb R^N$ be given.
    Consider the following elliptic problem:
 \begin{equation}\label{elsv}
  \begin{cases}
-\Delta\psi=\omega  &\mbox{in } D,\\
\psi=0 &\mbox{on } \Gamma_{0},\\
\psi \mbox{ is constant on }\Gamma_{i},& i=1,\cdot\cdot\cdot,N,\\
\int_{\Gamma_{i}}\nabla \psi\cdot d \mathbf n=-a_{i},& i=1,\cdot\cdot\cdot,N.
\end{cases}
 \end{equation}
  Then there is a unique $\psi\in H^2(D)$ solving \eqref{elsv}, given by
  \[\psi=\mathcal P\omega+h_{\mathbf a}.\]

  \end{proposition}
  \begin{proof}

We first prove uniqueness. Suppose that \eqref{elsv} has two solutions, say  $\psi_{1}$ and $\psi_{2}. $ It is clear that
\[\psi_{1}-\psi_{2}\in \mathcal Y,\quad\Delta(\psi_{1}-\psi_{2})=0.\]
 Applying \eqref{ypx99} we have
 \[ \int_{D}|\nabla (\psi_1-\psi_2)|^{2}dx=0,\]
which implies that $\psi_1-\psi_2$ is constant in $D$ (since $D$ is connected).  Taking into account the fact that both $\psi_1$ and $\psi_2$ vanish on the outer boundary component $\Gamma_0$, we get $\psi_1\equiv \psi_2.$

Now we complete the proof  by showing that $\psi=\mathcal P\omega+h_{\mathbf a}$ solves \eqref{elsv}. The only thing we need to verify is that
   \begin{equation}\label{ontves}
   \int_{\Gamma_{i}}\nabla(\mathcal P\omega+h_{\mathbf a})\cdot d\mathbf n=-a_{i}.
 \end{equation}
This follows from the following direct computation:
    \begin{align*}
   \int_{\Gamma_{i}}\nabla(\mathcal P\omega+h_{\mathbf a})\cdot d\mathbf n
   &=\int_{\Gamma_{i}}\nabla \mathcal G\omega\cdot d\mathbf n+\sum_{j,k=1}^{N}q_{jk}\left(\int_{D}\zeta_{j}\omega dx- a_{j} \right)\int_{\Gamma_{i}}\nabla\zeta_{k}\cdot d\mathbf n\\
   &=\int_{\Gamma_{i}}\nabla \mathcal G\omega\cdot d\mathbf n+\sum_{j,k=1}^{N}q_{jk}\left(\int_{D}\zeta_{j}\omega dx- a_{j} \right)\int_{\partial D}\zeta_{i}\nabla\zeta_{k}\cdot d\mathbf n\\
   &=\int_{\Gamma_{i}}\nabla \mathcal G\omega\cdot d\mathbf n+\sum_{j,k=1}^{N}q_{jk}\left(\int_{D}\zeta_{j}\omega dx- a_{j} \right)\int_{  D}\nabla\zeta_{i}\cdot\nabla\zeta_{k}dx\\
   &=\int_{\Gamma_{i}}\nabla \mathcal G\omega\cdot d\mathbf n+\sum_{j,k=1}^{N}p_{ik}q_{jk}\left(\int_{D}\zeta_{j}\omega dx- a_{j} \right)\\
   &=\int_{\Gamma_{i}}\nabla \mathcal G\omega\cdot d\mathbf n+ \int_{D}\zeta_{i}\omega dx- a_{i} \\
   &=-a_{i}.
 \end{align*}
Note that in the last equality we used the following fact
 \[\int_{\Gamma_{i}}\nabla \mathcal G\omega\cdot d\mathbf n=-\int_{D}\zeta_{i}\omega dx,\]
 which is due to the  second equality in \eqref{yp00}.

 \end{proof}

{\bf Acknowledgements:}
{G. Wang was supported by National Natural Science Foundation of China (12001135, 12071098) and China Postdoctoral Science Foundation (2019M661261, 2021T140163). B. Zuo was supported by National Natural Science Foundation of China (12101154).}

\phantom{s}
 \thispagestyle{empty}


\begin{thebibliography}{99}
\bibitem{Abe}
K. Abe and K. Choi,  Stability of Lamb dipoles, \textit{Arch. Ration. Mech. Anal.}, 244(2022), 877--917.
\bibitem{A1}
V. I. Arnold,  Conditions for nonlinear stability plane curvilinear flow of an idea fluid,
\textit{Sov. Math. Dokl.}, 6(1965), 773--777.

\bibitem{A11}
V. I. Arnold, Variational principles for three-dimensional steady-state flows of an ideal fluid, \textit{J. Appl. Math. Mech.}, 29(1965), 1002--1008.

\bibitem{A2}
V. I. Arnold,  On an a priori estimate in the theory of hydrodynamical stability, \textit{Amer. Math. Soc. Transl.}, 79(1969), 267--269.

\bibitem{AK}
V. I. Arnold and B. A. Khesin, Topological methods in hydrodynamics, 2nd ed.,  Applied Mathematical Sciences 125,  Springer, Cham,  2021.
\bibitem{Bad}
 M. Badiale and E. Serra,   Semilinear elliptic equations for beginners--Existence results via the variational approach. Universitext. Springer, London, 2011.

\bibitem{BG}
C. Bardos, Y. Guo, W. Strauss,  Stable and unstable ideal plane flows, \textit{ Chinese Ann. Math. Ser. B}, 23(2002),149--164.

\bibitem{BD}
J. Beichman and S. Denisov,  2D Euler equation on the strip: stability of a rectangular patch, \textit{Comm. Partial Differential Equations,} 42(2017), 100--120.

\bibitem{Bre}
 H. Brezis,  Functional analysis, Sobolev spaces and partial differential equations, Universitext, Springer, New York, 2011.

\bibitem{B1}
G. R. Burton, Rearrangements of functions, maximization of convex functionals, and vortex rings, \textit{Math. Ann.}, 276(1987), 225--253.

\bibitem{B2}
G. R. Burton, Variational problems on classes of rearrangements and multiple configurations for steady vortices, \textit{Ann. Inst. H. Poincar\'e. Anal. Non Lin\'eare.}, 6(1989), 295-319.
\bibitem{B5}
G. R. Burton, Global nonlinear stability for steady ideal fluid flow in bounded planar domains,
\textit{Arch. Ration. Mech. Anal.}, 176(2005), 149--163.



\bibitem{Bjde}
G. R. Burton, Compactness and stability for planar vortex-pairs with prescribed impulse, \textit{J. Differential Equations}, 270(2021), 547--572.


\bibitem{BM}G. R. Burton and J. B. McLeod, Maximisation and minimisation on classes of rearrangements. \textit{Proc. Roy. Soc. Edin. Sect. A}, 119(1991), 287--300.

\bibitem{Bcmp}
G. R. Burton, H. J. Nussenzveig Lopes and M. C. Lopes Filho, Nonlinear stability for steady vortex pairs, \textit{Comm. Math. Phys.}, 324(2013), 445--463.

\bibitem{CWW}
 D. Cao,  J. Wan and G. Wang, Nonlinear orbital stability for planar vortex patches, \textit{Proc. Amer. Math. Soc.}, 147(2019),  775--784.
\bibitem{CWCV}
D. Cao and G. Wang, Steady vortex patches with opposite rotation directions in a planar ideal fluid, \textit{Calc. Var. Partial Differential Equations,} 58(2019), Paper No. 75.


\bibitem{CWN}
D. Cao and G. Wang, Nonlinear stability of planar vortex patches in an ideal fluid, \textit{J. Math. Fluid Mech.}, 23(2021), Paper No. 58.






 \bibitem{CD}
 K. Choi and D.  Lim, Stability of radially symmetric, monotone vorticities of 2D Euler equations, \textit{Calc. Var. Partial Differential Equations}, 61(2022),  Paper No. 120.




\bibitem{LCE}
L. C. Evans, Partial Differential Equations, 2nd ed., American Mathematical Society, Providence, RI, 2010.


 \bibitem{FH}
 S. Friedlander and L. Howard,
Instability in parallel flows revisited, \textit{Stud. Appl. Math.}, 101(1998),  1--21.

 \bibitem{FSV}
 S. Friedlander, W. Strauss and M. Vishik,
Nonlinear instability in an ideal fluid, \textit{Ann. Inst. H. Poincar\'e. Anal. Non Lin\'eare.},14(1997),  187--209.




 \bibitem{G}
 E. Grenier,  On the nonlinear instability of Euler and Prandtl equations, \textit{Comm. Pure Appl. Math.}, 53(2000),  1067--1091.
\bibitem{K}
H. Koch, Transport and instability for perfect fluids. \textit{Math. Ann.}, 323(2002), 491--523.







 \bibitem{LL} E. H. Lieb and M. Loss, Analysis, Second edition,\textit{ Graduate Studies in Mathematics, Vol. 14}. American Mathematical Society, Providence, RI (2001).

\bibitem{LZ1}
Z. Lin, Instability of some ideal plane flows. \textit{SIAM J. Math. Anal.}, 35(2003), 318--356.
\bibitem{LZ} Z. Lin, Some stability and instability criteria for ideal plane flows. \textit{Comm. Math. Phys.}, 246(2004), 87--112.
\bibitem{LZ3}
 Z. Lin, Nonlinear instability of ideal plane flows. \textit{Int. Math. Res. Not.}, 41(2004), 2147--2178.






\bibitem{MB}
A. J. Majda and A. L. Bertozzi, Vorticity and incompressible flow, \textit{Cambridge Texts in Applied Mathematics, Vol. 27}, Cambridge University Press, 2002.



\bibitem{MPu}
C. Marchioro and M. Pulvirenti, Mathematical theory of incompressible noviscous fluids, Springer-Verlag, 1994.


 \bibitem{SVega}
T. C. Sideris and L. Vega,  Stability in $L^{1}$ of circular vortex patches, \textit{Proc. Amer. Math.
Soc.}, 137(2009), 4199--4202.



\bibitem{Ta}
Y. Tang,  Nonlinear stability of vortex patches,
 \textit{Trans. Amer. Math. Soc.},  304(1987),   617--637.


\bibitem{VF}
M. Vishik and S. Friedlander, Nonlinear instability in two dimensional ideal fluids: the case of a dominant eigenvalue, \textit{Comm. Math. Phys.},  243 (2003), 261--273.



 \bibitem{WP}
Y.-H. Wan and M. Pulvirenti, Nonlinear stability of circular vortex patches,
\textit{Comm. Math. Phys.}, 99(1985), 435--450.



 \bibitem{WGu1}
G. Wang,  Nonlinear stability of planar steady Euler flows associated with semistable solutions of elliptic problems, \textit{Trans. Amer. Math. Soc.}, 375(2022),  5071--5095.

 \bibitem{WGu2}
G. Wang, Stability of 2D steady Euler flows related to least energy solutions of the Lane-Emden equation, arXiv:2104.12406.


\bibitem{WG0}
G. Wolansky, M. Ghil, An extension of Arnol'd's second stability theorem for the Euler equations, \textit{Phys. D}, 94(1996), 161--167.


\bibitem{WG}
G. Wolansky, M. Ghil,
Nonlinear stability for saddle solutions of ideal flows and symmetry breaking. \textit{Comm. Math. Phys.}, 193(1998), 713--736.



\end{thebibliography}
\end{document}